\documentclass[11pt, leqno]{article}

\usepackage{amsmath,amsthm,amsfonts,amssymb,tabularx}
\usepackage[latin1]{inputenc}
\usepackage{graphicx}
\usepackage{color}
\usepackage{nameref}
\usepackage{pgf,tikz}
\usepackage{pdfsync}

\usetikzlibrary{arrows}

\usepackage[inner=3cm,outer=4.5cm,bottom=4cm,top=4cm]{geometry}

\newtheorem{theorem}{Theorem}[section]
\newtheorem{proposition}[theorem]{Proposition}

\newtheorem{defn}[theorem]{Definition}
\newtheorem{lemma}[theorem]{Lemma}
\newtheorem{cor}[theorem]{Corolary}
\newcommand\R{\mathbb{R}}

\newcommand\F{\mathcal{F}}
\newcommand\KK{\mathcal{K}}
\newcommand\Ks{\mathcal{K}_s}
\newcommand\HH{\mathcal{H}}
\newcommand\Hs{\mathcal{H}_s}

\newcommand\lapal{(-\Delta)^{\alpha}}

\newcommand\lapp{(-\Delta)^{-s}}

\newcommand\into{\int_{0}^T\int_{\mathbb{R}^N}}

\newcommand\intr{\int_{\mathbb{R}^N}}

\newcommand\x{\overline{x}}
\newcommand\lra{\longrightarrow}

\numberwithin{equation}{section}

\newcommand{\RN}{\mathbb{R}^N}

\begin{document}

\title{\bf Finite and infinite speed of propagation for  \\ porous medium equations with nonlocal pressure}

\author{\bf Diana Stan, F\'elix del Teso  and Juan Luis V\'azquez}

\maketitle

\begin{abstract}
We study a porous medium equation with fractional potential pressure:
$$
\partial_t u= \nabla \cdot (u^{m-1} \nabla p), \quad  p=(-\Delta)^{-s}u,
$$
for $m>1$, $0<s<1$ and $u(x,t)\ge 0$.  The problem is posed for $x\in \mathbb{R}^N$, $N\geq 1$, and $t>0$. The initial data $u(x,0)$ is assumed to be a bounded function with compact support or fast decay at infinity. We establish existence of a class of weak solutions for which we determine whether the property of compact support is conserved in time depending on the parameter $m$, starting from the  result of finite propagation known for $m=2$. We find that when $m\in [1,2)$ the problem has infinite speed of propagation, while for $m\in [2,3)$ it has finite speed of propagation. In other words $m=2$ is critical exponent regarding propagation. The main results have been announced in the note \cite{StanTesoVazCRAS}.
\end{abstract}

\vspace{1.3cm}
\vspace{0.3cm}


\vspace{1.5cm}
\noindent \textbf{Addresses:}\\
Diana Stan, {\tt diana.stan@uam.es}, \\F\'{e}lix del Teso, {\tt felix.delteso@uam.es},
\\and Juan Luis V{\'a}zquez, {\tt juanluis.vazquez@uam.es},\\
Departamento de Matem\'{a}ticas, Universidad
Aut\'{o}noma de Madrid, \\
Campus de Cantoblanco, 28049 Madrid, Spain

\

\newpage
\small
\tableofcontents
\newpage

\section{Introduction}

In this paper we study the following nonlocal evolution equation
\begin{equation}\label{model1}
   \left\{ \begin{array}{ll}
  u_{t}(x,t) = \nabla \cdot (u^{m-1} \nabla p), \quad  p=(-\Delta)^{-s}u,   &\text{for } x \in \RN, \, t>0,\\[2mm]
  u(0,x)  =u_0(x) &\text{for } x \in \RN,
    \end{array}  \right.
\end{equation}
for $m>1$ and $u(x,t)\ge 0$. The model formally resembles the classical \emph{Porous Medium Equation} (PME) $u_t=\Delta u^m = \nabla (mu^{m-1}\nabla u)$ where the pressure $p$ depends linearly on the density function $u$ according to the Darcy Law. In this model the pressure $p$ takes into consideration nonlocal effects through the Inverse Fractional Laplacian operator $\KK_s=(-\Delta)^{-s}$, that is the Riesz potential of order $2s$. The problem is posed for $x \in \RN$, $N\ge 1$ and $t>0$. The initial data $u_0:\RN \to [0,\infty) $  is bounded with compact support or fast decay at infinity.

As a motivating precedent, in the work \cite{CaffVaz} Caffarelli and V\'{a}zquez proposed the following model of porous medium equation with nonlocal diffusion effects
\[
  \partial_t u= \nabla \cdot (u \nabla p), \ \ \ p=(-\Delta)^{-s}u. \tag{CV}\label{ModelCaffVaz}
\]
The study of this model has been performed in a series of papers as follows. In \cite{CaffVaz}, Caffarelli and V\'{a}zquez developed the theory of existence of bounded weak solutions that propagate with finite speed. In \cite{CaffVaz2}, the same authors proved the asymptotic time behaviour of the solutions. Self-similar non-negative solutions are obtained by solving an elliptic obstacle problem with fractional Laplacian for the pair pressure-density, called obstacle Barenblatt solutions. Finally, in \cite{CaffSorVaz}, Caffarelli, Soria and V\'{a}zquez considered the regularity and the $L^1-L^\infty$ smoothing effect. The regularity for $s=1/2$ has been recently done in \cite{CaffVazquezReg}. The study of fine asymptotic behaviour (rates of convergence) for \eqref{ModelCaffVaz} has been performed by Carrillo, Huang, Santos and V\'azquez \cite{CarrilloHuangVazquez} in the one dimensional setting.
Putting $m=2$ in \eqref{model1}, we recover Problem \eqref{ModelCaffVaz}.

A main question in this kind of nonlocal nonlinear diffusion models is to decide whether compactly supported data produce compactly supported solutions, a property known as finite speed of propagation. Surprisingly, the answer was proved to be positive for $m=2$ in paper \cite{CaffVaz}, for $m=1$ we get the linear fractional heat equation, that  is explicitly solvable by convolution with a positive kernel, hence it has infinite speed of propagation.
The main motivation of this paper is establishing the alternative finite/infinite speed of propagation for the solutions of Problem \eqref{model1} depending on the parameter $m$. In the process we construct a theory of existence of solutions and derive the main properties.  A modification of the numerical methods developed in \cite{Teso,TesoVaz} pointed to us to the possibility of having two different propagation properties.

\medskip

\noindent {\bf Other related models.} Equation (CV) with $s=1/2$ in dimension $N=1$ has been proposed by Head \cite{Head} to describe the dynamics of dislocation in crystals. The model is written in the integrated form as
$$
v_t+|v_x|(-\partial^2/\partial_{xx})^{1/2}v=0.
$$
The dislocation density is $u=v_x$. This model has been recently studied by Biler, Karch and Monneau in \cite{BilerKarchMonneau}, where they prove that the problem enjoys the properties of uniqueness and comparison of viscosity solutions. The relation between $u$ and $v$ is very interesting and will be used by us in the final sections.

Another possible generalization of the (CV) model  is
$$\partial_t u= \nabla \cdot (u \nabla p),\quad p=(-\Delta)^{-s}(|u|^{m-2}u),$$ that has been investigated by Biler, Imbert and Karch in \cite{BilerImbertKarchCRAS,BilerImbertKarch}. They prove the existence of weak solutions and they find explicit self-similar solutions with compact support for all $m>1$. The finite speed of propagation for every weak solution has been done in \cite{Imbert}.

The second nonlocal version of the classical PME is the model
$$
u_t=-(-\Delta)^{s'} u^{m'}, \quad m'>0,
$$
 known as the \emph{Fractional Porous Medium Equation} (FPME). This model has infinite speed of propagation and the existence of fundamental solutions  of self-similar type or Barenblatt solutions is known for  $m>(N-2s')_+/N$. We refer to the recent works \cite{PQRV1,PQRV2,VazquezBarenblattFractPME,BV2012}.
The (FPME) model for $m'=1$, also called linear fractional Heat Equation, coincides with model \eqref{model1} for $s=1-s'$, $m=1$.

\subsection{Main results}

We first propose a definition of solution and establish the existence and main properties of the solutions.

\begin{defn}\label{defWeakSolPMFP}
Let $m>1$. We say that $u$ is a weak solution of $\eqref{model1}$ in $Q_T=\RN \times (0,T)$ with nonnegative initial data $u_0\in L^1(\RN)$ if \ (i)  $u\in L^1(Q_T)$,
 (ii) $\nabla \KK_s[u] \in L^1([0,T): L^1_{loc}(\RN))$, \
(iii) $u^{m-1}\nabla \KK_s[u] \in L^1(Q_T)$, and (iv)
\begin{equation}\label{model1weak}
\into u \phi_tdxdt -\into u^{m-1}\nabla \KK_s (u) \nabla \phi dxdt+ \intr u_0(x)\phi(x,0)dx=0
\end{equation}
holds for every test function $\phi$ in $Q_T$ such that $\nabla \phi$ is continuous, $\phi$ has compact support in $\RN$ for all $t\in (0,T)$ and vanishes near $t=T$.
\end{defn}

Before entering the discussion of finite versus infinite propagation, we study the question of existence. We have the following  result for $1<m<2$.
\begin{theorem}\label{Thm1PMFP}
Let  $m\in (1,2)$, $N\ge 1$. Let $u_0\in L^1(\RN)\cap L^\infty (\RN)$.
Then there exists a weak solution $u$ of equation \eqref{model1} with initial data $u_0$ such that $u\in L^1(Q_T)\cap L^\infty (Q_T)$ and $\nabla\mathcal{H}_s[u] \in L^2(Q_T)$. Moreover, $u$ has the following properties:
\begin{enumerate}
\item \textbf{(Conservation of mass)} For all $t>0$ we have
$\displaystyle{
\int_{\RN}u(x,t)dx=\int_{\RN}u_0(x)dx.
}$
\item \textbf{($L^{\infty}$ estimate) } For all $t>0$ we have  $||u(\cdot,t)||_\infty\leq ||u_0||_\infty$.

\item \textbf{(First Energy estimate)} For all $t>0$,
$$
C\int_0^t  \int_{\RN} |\nabla \HH_s[u]|^2dxdt  +\int_{\RN}u(t)^{3-m}dx \le \int_{\RN}u_0^{3-m}dx,
$$
with $C=(2-m)(3-m)>0$.
\item \textbf{(Second Energy estimate)} For all $t>0$,
\begin{equation*}
\frac{1}{2}\int_{\RN}|\HH_s[u]|^2dx+ \int_0^t\int_{\RN}u^{m-1}|\nabla\KK_s[u]|^2dx\leq \frac{1}{2}\int_{\RN}|\HH_s[\widehat{u}_0]|^2dx .
\end{equation*}

\end{enumerate}
\end{theorem}
\noindent The existence for $m\geq2$ is covered in the following result.

\begin{theorem}\label{Thm2PMFP}
Let $m\in [2,3)$. Let $u_0\in L^1(\RN)\cap L^\infty (\RN)$ be such that
\begin{equation}\label{decayu0}
0\le u_0(x)\le Ae^{-a|x|} \text{ for some }A,a>0.
\end{equation}
Then there exists a weak solution $u$ of equation \eqref{model1} with initial data $u_0$
such that $u\in L^1(Q_T)\cap L^\infty (Q_T)$, $\nabla\mathcal{H}_s[u] \in L^2(Q_T)$  and $u$ satisfies the properties $1,2,4$ of Theorem \ref{Thm1PMFP}.  Moreover, the solution
decays exponentially in $|x|$ and the first energy estimate holds in the form
\begin{eqnarray*}
 |C|\int_0^t \int_{\RN} |\nabla \HH_s[u]|^2dxdt
+\int_{\RN}u_0^{3-m}dx \le \int_{\RN}u(t)^{3-m}dx
\end{eqnarray*}
 where $C=C(m)=(2-m)(3-m)$.
\end{theorem}

We should have covered existence in the whole range $m\geq2$ where we want to prove finite speed of propagation for the constructed weak solutions, see Theorem \ref{ThmInfFiniteProp}. But the existence theory used in the previous theorem breaks down because of the negative exponents $3-m$ that would appear in the first energy estimate for $m>3$ (a logarithm would appear for $m=3$). A new existence approach avoiding such estimate is needed,  and this can be done but is not immediate. We have refrained from presenting such a study here because it would divert us too much from the main interest.

The following is our most important contribution, which deals with the property of finite propagation of the solutions depending on the value of $m$.

\begin{theorem}\label{ThmInfFiniteProp}
a) Let $N\ge 1$, $m \in [2,3)$, $s\in (0,1)$  and let $u$ be a constructed weak solution to problem \eqref{model1} as in Theorem \ref{Thm2PMFP} with compactly supported initial data $u_0 \in L^1(\RN)\cap L^\infty(\RN)$.  Then,  $u(\cdot,t)$ is also compactly supported for any $t>0$, i.e. the solution has \textbf{finite speed of propagation}.

b) Let $N=1$, $m\in (1,2)$, $s\in (0,1)$  and let $u$ be a constructed solution as in Theorem \ref{Thm1PMFP}. Then for any $t>0$ and any $R>0$, the set $\mathcal{M}_{R,t}=\{x: |x|\ge R,\  u(x,t)>0\}$ has positive measure even if $u_0$ is compactly supported. This is a  weak form of  \textbf{infinite speed of propagation}. If moreover $u_0$ is radially symmetric and monotone non-increasing in $|x|$, then we get a clearer result:  $u(x,t)>0$ for all $x\in \R$ and $t>0$.
\end{theorem}

\noindent \textbf{Remark}

\noindent (i) By constructed weak solution we mean that it is the limit of the approximations process that produces the result of Theorem \ref{Thm2PMFP}.

\noindent (ii) We point out that  part (a) of the theorem would still be true when $m\geq 3$ once we supply an existence theory based on approximations with  solutions of regularized problems.

\subsection*{Organization of the proofs}

\noindent$\bullet$ In Section \ref{SectionEstim} we derive useful energy estimates valid for all $m>1$. Due to the differences in the computations, we will separate the cases $m \ne 2,3$ and $m=3$.

\noindent$\bullet$ In Section \ref{SectionExistI}, \ref{SectExponTail} and \ref{limitssec} we prove the existence of a weak solution of Problem \eqref{model1} as the limit of a sequence of solutions to suitable approximate problems. The range of exponents is $1<m<3$.

\noindent$\bullet$ Section \ref{SectionFiniteProp} deals with the property of finite speed of propagation for $m\ge 2$. See Theorem \ref{ThmBarrierParabola}.

\noindent$\bullet$ In Section \ref{SectionInfinite} we prove the infinite speed of propagation for $m\in (1,2)$ in the one-dimensional case. This section introduces completely different tools. Indeed, we develop a theory of viscosity solutions for the integrated equation $v_t + |v_x|^{m-1}(-\Delta)^{1-s}v=0$, where $v_x=u$ the solution of \eqref{model1}, and we  prove infinite speed of propagation in the usual sense for the solution $v$ of the integrated problem.

Though we do not get the same type of infinite propagation result for $1<m<2$  in several spatial dimensions, the evidence (partial results and explicit solutions) points in that direction, see the comments in Section \ref{sec.comm}.

\section{Functional setting} \label{subsecFunctSett}

We will work with the following functional spaces (see \cite{Hitch2012}).
Let $s\in (0,1)$.  Let $\mathcal{F}$ denote the Fourier transform. We consider
$$
H^s(\RN)=\left\{u:L^2(\RN): \int_{\RN}(1+|\xi|^{2s})|\mathcal{F}u(\xi)|^2d\xi < +\infty \right\}
$$
with the norm
$$
\|u\|_{H^s(\RN)}=\|u\|_{L^2(\RN)} + \int_{\RN} |\xi|^{2s}|\mathcal{F}u(\xi)|^2d\xi.
$$
For functions $u \in H^s(\RN)$, the Fractional Laplacian is defined by
$$
(- \Delta)^s u (x) =C_{N,s}\,\text{P.V.}\int_{\RN}\frac{u(x)-u(y)}{|x-y|^{N+2s}}dy =  C \mathcal{F}^{-1}(|\xi|^{2s}(\mathcal{F}u)),
$$
where $C_{N,s}=\pi^{-(2s+N/2)}\Gamma(N/2+s)/\Gamma(-s).$
  Then
$$
\|u\|_{H^s(\RN)}=\|u\|_{L^2(\RN)} +  C \| (-\Delta)^{s/2}u\|_{L^2(\RN)}.
$$
For functions $u$ that are defined on a subset $\Omega \subset \RN$ with $u=0$ on the boundary $\partial \Omega$, the fractional Laplacian and the $H^s(\RN)$ norm are computed by extending the function $u$ to all $\RN$ with $u=0$ in $\RN \setminus \Omega.$
For technical reasons we will only consider the case $s<1/2$ in $N=1$ dimensional space.

The inverse operator $(-\Delta)^{-s}$ coincides with the Riesz potential of order $2s$ that will be denoted here by $\Ks$. It can be represented by convolution with the Riesz kernel $K_s$:
\[
\Ks[u]=K_s*u, \ \ \ K_s(x)=\frac{1}{c(N,s)}|x|^{-(N-2s)},
\]
where $c(N,s)=\pi^{N/2-2s}\Gamma(s)/\Gamma((N-2s)/2).$ The Riesz potential $\Ks$ is a self-adjoint operator. The square root of $\Ks$ is $\KK_{s/2}$, i.e. the Riesz potential of order $s$ (up to a constant). We will denote it by $\Hs:=(\Ks)^{1/2}$. Then $\Hs$ can be represented by convolution with the kernel $K_{s/2}$. We will write $\KK$ and $\HH$ when $s$ is fixed and known. We refer to \cite{landkof} for the arguments of potential theory used throughout the paper.

The inverse fractional Laplacian $\Ks[u]$ is well defined as an integral operator for all $s\in (0,1)$ in dimension $N\ge 2$, and $s\in (0,1/2]$ in the one-dimensional case $N=1$. We extend our result to the remaining case $s\in (1/2,1)$ by giving a suitable meaning to the combined operator $(\nabla \Ks)$. The details concerning this case will be given in Section \ref{SubsectN1}.

For functions depending on $x$ and $t$, convolution is applied for fixed $t$ with respect to the spatial variables and we then write $u(t)=u(\cdot,t)$.

\subsection{Functional inequalities related to the fractional Laplacian}

We recall some functional inequalities related to the fractional Laplacian operator that we used throughout the paper. We refer to \cite{PQRV2} for the proofs.

\begin{lemma}[\textbf{Stroock-Varopoulos Inequality}] Let $0<s<1$, $q>1$. Then
\begin{equation}\label{StroockVar}
\int_{\RN}|v|^{q-2}v (-\Delta)^{s}v dx \ge \frac{4(q-1)}{q^2}\int_{\RN}\left| (-\Delta)^{s/2}|v|^{q/2}\right|^2 dx
\end{equation}
for all $v\in L^q(\RN)$ such that $(-\Delta)^{s}v \in L^q(\RN)$.
\end{lemma}

\begin{lemma}[\textbf{Generalized Stroock-Varopoulos Inequality}] Let $0<s<1$. Then
\begin{equation}\label{StroockVar2}
\int_{\RN}\psi(v)(-\Delta)^{s}v dx \ge \int_{\RN}\left| (-\Delta)^{s/2}\Psi(v)\right|^2 dx
\end{equation}
whenever $\psi'=(\Psi')^2$.
\end{lemma}

\begin{theorem}[\textbf{Sobolev Inequality}] Let $0<s<1$ ($s<\frac{1}{2}$ if $N=1$). Then
\begin{equation}\label{SobolevIneq}
\|f\|_{\frac{2N}{N-2s}}\le \mathcal{S}_s \left\| (-\Delta)^{s/2}f\right\|_2,
\end{equation}
where the best constant is given in \cite{BV2012} page 31.
\end{theorem}

\subsection{Approximation of the Inverse Fractional Laplacian $(-\Delta)^{-s}$}\label{aproxinvlap}
We consider an approximation $\KK_s^\epsilon$ as follows. Let $K_s(z)=c_{N,s}|z|^{-(N-2s)}$ the kernel of the Riesz potential $\KK_s=(-\Delta)^{-s}$, $0<s<1$ ($0<s< 1/2$ if $N=1$). Let $\rho_\epsilon(x)=\epsilon^{-N} \rho(x/\epsilon)$, $\epsilon>0$ a standard mollifying sequence, where $\rho$ is positive, radially symmetric and decreasing, $\rho \in C_c^\infty(\RN)$ and $\intr \rho \, dx=1$. We define the regularization of $K_s$ as $K_s^\epsilon=\rho_\epsilon \star K_s$. Then
\begin{equation}\label{AproxInverFracLap}\mathcal{K}_s^\epsilon [u] = K_s^\epsilon \star u
\end{equation}
is an approximation of the Riesz potential $\mathcal{K}_s=(-\Delta)^{-s}$. Moreover, $\mathcal{K}_s$ and $\mathcal{K}_s^\epsilon$ are self-adjoint operators with $\mathcal K_s=(\mathcal H_s)^2$, $\mathcal{K}_s^\epsilon=(\mathcal{H}_s^\epsilon)^2$.
Also, $\rho=\sigma*\sigma$ where $\sigma$ has the same properties as $\rho$. Then, we can write $\HH_s^\epsilon$ as the operator with kernel $K_{s/2}*\sigma_\epsilon$. That is:
$$
\intr u \ \mathcal{K}_s^\epsilon [u] dx= \intr |\mathcal{H}_s^\epsilon [u]|^2 dx.
$$
Also $\mathcal{H}_s^\epsilon $ commutes with the gradient:
$$
\nabla \mathcal{H}_s^\epsilon [u] = \mathcal{H}_s^\epsilon [\nabla u].
$$

\section{Basic estimates}\label{SectionEstim}

In what follows, we perform formal computations on the solution of Problem \eqref{model1}, for which we assume smoothness, integrability and fast decay as $|x|\rightarrow \infty$. The useful computations for the theory of existence and propagation will be justified later by the approximation process.
We fix $s\in (0,1)$ and $m\ge 1$. Let $u$ be the solution of Problem \eqref{model1} with initial data $u_0\ge 0$. We assume $u\geq0$ for the beginning. This property will be proved later.

\noindent $\bullet$ \textbf{Conservation of mass:}
 \begin{equation}\label{ConsMass}
 \frac{d}{d t} \intr u(x,t)dx=\intr u_tdx= \intr \nabla \cdot (u^{m-1} \nabla \KK_s [u])dx=0.
\end{equation}
\noindent$\bullet$ \textbf{First energy estimate:}
The estimates here are significantly different depending on the exponent $m$. Therefore, we consider the cases:

\noindent{\sc Case} $m=3$:
\begin{eqnarray*}
 \frac{d}{d t} \intr \log u(x,t)dx&=&\intr \frac{u_t}{u}dx= \intr \nabla u\cdot \nabla \KK_s[u]=\intr |\nabla\HH_s[u]|^2dx.
\end{eqnarray*}
Therefore, by the conservation of mass \eqref{ConsMass} we obtain
\begin{equation}
\frac{d}{d t} \intr( u- \log u)dx=-\intr |\nabla\HH_s[u]|^2dx.
\end{equation}

\noindent{\sc Case} $m\not= 3$:
\begin{eqnarray}
 \frac{d}{d t} \intr u^{3-m}(x,t)dx&= &(3-m)\intr u^{2-m}u_t dx=(3-m)\intr u^{2-m} \nabla (u^{m-1}\nabla \KK_s[u])dx\nonumber\\
 &=&-(3-m)(2-m)\intr \nabla u\cdot \nabla \KK_s[u]dx=-C\intr  |\nabla\HH_s[u]|^2dx\nonumber.
 \end{eqnarray}
Here $C=(3-m)(2-m)$ is negative for $m\in(2,3)$ and positive otherwise.

\noindent If $m>3$ or $1<m<2$ then
 \begin{equation*}
 \frac{d}{dt}\int_{\RN} u^{3-m}dx=-|C|\intr  |\nabla\HH_s[u]|^2dx.
 \end{equation*}

\noindent If $2<m<3$ then
 \begin{equation*}
 \frac{d}{dt}\int_{\RN} u^{3-m}dx=|C|\intr  |\nabla\HH_s[u]|^2dx,
 \end{equation*}
 or equivalently
  \begin{equation*}
\frac{d}{dt}\int_{\RN} u- u^{3-m}dx=-| C|\intr  |\nabla\HH_s[u]|^2dx.
 \end{equation*}

\noindent$\bullet$ \textbf{Second energy estimate:}
\begin{eqnarray}\label{SecondEnergy}
\frac{1}{2} \frac{d}{d t} \intr |\HH_s [u](x,t)|^2dx=\intr \HH_s [u] (\HH_s [u])_tdx=\intr \KK_s [u] u_tdx\\
=\intr \KK_s [u] \nabla \cdot (u^{m-1} \nabla \KK_s [u])dx=-\intr u^{m-1}|\nabla \KK_s [u]|^2dx. \nonumber
\end{eqnarray}

\noindent$\bullet$ $L^\infty$ \textbf{estimate: } We prove that the $L^\infty(\mathbb{R}^N)$ norm does not increase in time. Indeed, at a point of maximum $x_0$ of $u$ at time $t=t_0$, we have
\[
u_t=(m-1)u^{m-1}\nabla u \cdot \nabla p+ u^{m-1}\Delta \KK_s[u].
\]
The first term is zero since $\nabla u(x_0,t_0)=0$. For the second one we have $-\Delta \KK_s =(-\Delta)\lapp=(-\Delta)^{1-s}$ so that
\[
\Delta \KK_s [u](x_0,t_0)=-(-\Delta)^{1-s} u(x_0,t_0) =-c\intr \frac{u(x_0,t_0)-u(y,t_0)}{|x_0-y|^{N-2(1-s)}}dy \leq 0,
\]
where $c=c(s,N)>0$. We conclude by the positivity of $u$ that
\[
u_t(x_0,t_0)=u^{m-1}(x_0,t_0)\Delta \KK_s [u](x_0,t_0) \leq 0.
\]

\noindent$\bullet$ \textbf{Conservation of positivity:} we prove that if $u_0\geq0$ then $u(t)\geq0$ for all times. The argument is similar to the one above.

\noindent$\bullet$ $L^p$ \textbf{estimates for $1<p<\infty$.} The following computations are valid for all $m\geq1$, since $p+m-2>0$:
\begin{eqnarray*}
& & \displaystyle\frac{d}{dt}\intr u^p(x,t)dx=p\intr u^{p-1} \nabla \cdot (u^{m-1} \nabla \KK_s [u])dx\nonumber\\
&= & \displaystyle-p\intr u^{m-1} \nabla (u^{p-1})\cdot \nabla \KK_s [u]dx=-\frac{p(p-1)}{m+p-2}\intr \nabla (u^{p+m-2})\cdot\nabla \KK_s [u]dx\nonumber\\
&=& \frac{p(p-1)}{m+p-2}\intr u^{p+m-2}  \Delta \KK_s [u]dx\nonumber=-\frac{p(p-1)}{m+p-2}\intr u^{p+m-2}  (-\Delta)^{1-s}u \ dx\\
&\le& -\frac{4p(p-1)}{(m+p-1)^2}\int_{\RN}\left|(-\Delta)^{\frac{1-s}{2}} u^{\frac{m+p-1}{2}}\right|^2 dx,
\end{eqnarray*}
where we applied the Stroock-Varopoulos inequality \eqref{StroockVar} with $r=m+p+1$. We obtain that $\displaystyle {\intr u^p(t)dx}$ is non-increasing in time. Moreover, by Sobolev Inequality \eqref{SobolevIneq} applied to the function $\displaystyle {f= u^{(m+p-1)/2}}$, we obtain that
\begin{eqnarray*}
& & \displaystyle\frac{d}{dt}\intr u^p(x,t)dx \le -\frac{4p(p-1)}{(m+p-1)^2 \mathcal{S}_{1-s}^2} \left( \int_{\RN}|u(x,t)|^{\frac{N(m+p-1)}{N-2+2s}}dx\right)^{\frac{N-2+2s}{N}},
\end{eqnarray*}
with the restriction of $s>1/2$ if $N=1$.

\section{Existence of smooth approximate solutions for $m\in (1,\infty)$}\label{SectionExistI}
Our aim is to solve the initial-value problem \eqref{model1} posed in $Q=\RN\times(0,\infty)$ or at least $Q_T=\RN\times(0,T)$, with parameter $0<s<1$. We will consider initial data $u_0\in L^1(\RN)$. We assume for technical reasons that $u_0$ is bounded and we also impose decay conditions as $|x|\to \infty$.

\subsection{Approximate problem}

We make an approach to problem \eqref{model1} based on regularization, elimination of the degeneracy and reduction of the spatial domain.  Once we have solved the approximate problems, we derive estimates that allow us to pass to the limit in all the steps one by one, to finally obtain the existence of a weak solution of the original problem \eqref{model1}. Specifically, for small $\epsilon, \delta, \mu \in (0,1)$ and $R>0$ we consider the following initial boundary value problem
posed in $Q_{T,R}=\{x\in B_R(0), \ t\in (0,T)\}$
\[
\left\{
\begin{array}{ll}
(U_1)_t= \delta \Delta U_1 +\nabla \cdot(d_\mu(U_1) \nabla \KK_s^\epsilon [U_1])&\text{for } (x,t)\in Q_{T,R},\\
U_1(x,0)=\widehat{u}_0(x) &\text{for } x \in B_R(0),\\
U_1(x,t)=0 &\text{for } x\in \partial B_R(0), \ t\geq 0.
\end{array}
\right.
\tag{$P_{\epsilon\delta\mu R}$}\label{model1Aprox} \]
The regularization tools that we use are as follows. $\widehat{u}_0=\widehat{u}_{0,\epsilon,R}$ is a nonnegative, smooth and bounded approximation of the initial data $u_0$ such that $\| \widehat{u}_0 \|_\infty\le \|u_0\|_\infty$ for all $\epsilon>0.$
For every $\mu >0$, $d_\mu : [0,\infty)\to [0,\infty) $ is a continuous function defined by
\begin{equation}\label{dmu}
d_\mu(v)=(v+\mu)^{m-1}.
\end{equation}
The approximation of $\mathcal{K}_s^{\epsilon}$ of $\mathcal{K}_s=(-\Delta)^{-s}$ is made as before in Section \ref{subsecFunctSett}.  The existence of a solution $U_1(x,t)$ to Problem \eqref{model1Aprox} can be done by classical methods and the solution is  smooth. See for instance \cite{LionsMasGalic} for similar arguments.

In the weak formulation we have
\begin{equation}\label{weaksolAprox}
\int_0^T\int_{B_R} U_1(\phi_t-\delta \Delta \phi)dxdt-\int_0^T\int_{B_R}  d_\mu(U_1) \nabla \KK_s^\epsilon [U_1]\nabla \phi dxdt+ \int_{B_R} \widehat{u}_0(x) \phi(x,0)dx=0
\end{equation}
 valid for smooth test functions $\phi$ that vanish on the spatial boundary $\partial B_R$ and for large $t$. We use the notation $B_R=B_R(0)$.

\noindent\textbf{Notations.} The existence of a weak solution of problem \eqref{model1} is done by passing to the limit step-by-step in the approximating problems as follows. We denote by $U_1$ the solution of the approximating problem \eqref{model1Aprox} with parameters $\epsilon,\delta,\mu,R$. Then we will obtain $U_2(x,t)=\lim_{\epsilon\to 0}U_1(x,t)$. Thus $U_2$ will solve an approximating problem \eqref{ProblemMuDeltaR} with parameters $\delta,\mu,R$. Next, we take $U_3=\lim_{R\to \infty}U_2(x)$ that will be a solution of Problem \eqref{ProblemMuDelta}, $U_4:=\lim_{\mu \to 0}U_3(x,t)$ solving Problem \eqref{ProblemDelta}. Finally we obtain $u(x,t)=\lim_{\delta \to 0}U_4(x,t)$ which solves problem \eqref{model1}.

\subsection{A-priori estimates for the approximate problem}\label{SectApriori}
 We derive suitable a-priori estimates for the solution $U_1(x,t)$  to Problem \eqref{model1Aprox} depending on the parameters $\epsilon,\delta,\mu, R$.

\medskip

\noindent$\bullet$\textbf{ Decay of total mass:} Since $U_1\geq 0$ and $U_1=0$ in $\partial B_R$, then $\displaystyle{\frac{\partial U_1}{\partial n}\leq 0}$ and so, an easy computation gives us
\begin{eqnarray}
\frac{d}{dt}\int_{B_R}U_1(x,t)dx&=& \delta \int_{B_R}\Delta U_1 \ dx+ \int_{B_R} \nabla\cdot (d_\mu(U_1)\nabla \KK_\epsilon [U_1])dx\nonumber\\
&=&\int_{\partial B_R} \frac{\partial U_1}{\partial n}d\sigma+\int_{\partial B_R} d_\mu(U_1)\frac{\partial (\KK_\epsilon [U_1])}{\partial n}d\sigma\leq0.
\end{eqnarray}
We conclude that
$$
\int_{B_R}U_1(x,t)dx \le \int_{B_R}\widehat{u}_0(x) \quad \text{ for all } t >0.
$$

\medskip

\noindent$\bullet$\textbf{ Conservation of $L^\infty$ bound}: we prove that $0\leq U_1(x,t)\leq ||\widehat{u}_0||_{\infty}$.
The argument is as in the previous section, using also that at a minimum point $\Delta U_1 \geq 0$ and at a maximum point $\Delta U_1 \leq 0$.  Also at this kind of points we have that
\[
\nabla d_\mu(U_1)=d_\mu'(U_1)\nabla U_1=0.
\]

\noindent$\bullet$\textbf{ Conservation of non-negativity}: $U_1(x,t)\ge 0$ for all $t>0$, $x\in B_R$. The proof is similar to the one in the previous section.
\medskip

\subsubsection{First energy estimate}

We choose a function $F_\mu$ such that
$$
F_{\mu}(0)=F'_{\mu}(0)=0 \quad \text{and} \quad F''_{\mu}(u)=1/d_\mu(u).
$$
Then, with these conditions one can see that $F_\mu (z) >0$ for all $z>0$. Also $F_\mu(U_1)$ and $F'_\mu(U_1)$ vanish on $\partial B_r \times [0,T]$, therefore, after integrating by parts, we get
\begin{equation}\label{derivFmu}
\frac{d}{dt}\int_{B_R}F_\mu(U_1)dx=-\delta \int_{B_R} \frac{|\nabla U_1|^2}{d_\mu(u)}dx- \int_{B_R} |\nabla \HH_s^\epsilon[U_1]|^2dx,
\end{equation}
where $\HH_\epsilon=\KK^{1/2}_\epsilon$. This formula implies that for all $0<t<T$ we have
\begin{equation}\label{energy1}
\int_{B_R} F_\mu(U_1(t))dx+\delta\int_0^t \int_{B_R} \frac{|\nabla U_1|^2}{d_\mu(U_1)}dxdt+\int_0^t  \int_{B_R} |\nabla   \HH_s^\epsilon[U_1]|^2dxdt=\int_{B_R}F_\mu(\widehat{u}_0)dx.
\end{equation}
This implies estimates for $|\nabla \HH_s^\epsilon(U_1)|^2$ and $\delta|\nabla U_1|^2/d_\mu(U_1)$ in $L^1 (Q_{T,R})$. We show how the upper bounds for such norms depend on the parameters $\epsilon,\delta, R, \mu$ and $T$.

The explicit formula for $F_\mu$  is as follows:
\begin{equation*}
F_\mu(U_1)=
\left\{
\begin{array}{ll}
 \displaystyle{\frac{1}{(2-m)(3-m)}[(U_1+\mu)^{3-m}-\mu^{3-m}]-\frac{1}{2-m}\mu^{2-m}U_1}&\text{for } m\not=2,3,\\[3mm]
-\log\left(1+(U_1/\mu  \right)) + U_1/\mu ,&\text{for }m=3,\\[3mm]
(U_1+\mu) \log\left(1+(U_1/\mu)\right)-U_1, &\text{for } m=2.
\end{array}
\right.
\end{equation*}
From formula \eqref{derivFmu} we obtain that the quantity $\int_{B_R}F_\mu(U_1(x,t))dx$ is non-increasing in time:
$$
0\le \int_{B_R}F_\mu(U_1(x,t))dx \le \int_{B_R}F_\mu(\widehat u_0)dx, \quad \forall t>0.
$$
Then, if we control the  term  $\int_{B_R}F_\mu(\widehat u_0)dx$, we will obtain uniform estimates independent of time $t>0$ for the quantity
$$
\delta\int_0^t \int_{B_R} \frac{|\nabla U_1|^2}{d_\mu(U_1)}dxdt + \int_0^t  \int_{B_R} |\nabla   \HH_s^\epsilon[U_1]|^2dxdt.
$$
These estimate are different depending on the range of parameters $m$.

\noindent$\bullet$\textbf{ Uniform bound in the case $m\in (1,2)$}.  We obtain uniform bounds in all parameters $\epsilon, R, \delta,\mu$ for the energy estimate \eqref{energy1}, that allow us to pass to the limit and obtain a solution of the original problem \eqref{model1}. By the Mean Value Theorem
\begin{align*}
\int_{B_R}F_\mu(\widehat u_0)dx &\le \frac{1}{(2-m)(3-m)}\int_{B_R}[(\widehat u_0+\mu)^{3-m}-\mu^{3-m}]dx \\
&\le \frac{1}{2-m}\int_{B_R}(\widehat u_0+\mu)^{2-m} \widehat u_0 dx    \\
&\le     \frac{1}{2-m}(\|u_0\|_{\infty}+1)^{2-m}\int_{\RN} u_0dx.
\end{align*}
Our main estimate in the case $m\in (1,2)$ is:
\begin{equation}\label{energy2}
\delta\int_0^t \int_{B_R} \frac{|\nabla U_1|^2}{d_\mu(U_1)} dx dt+\int_0^t  \int_{B_R} |\nabla \HH_s^\epsilon[U_1]|^2dxdt \le C_1,
\end{equation}
where $\displaystyle{C_1=C_1\left(m, u_0\right)= \frac{2}{(2-m)}(\|u_0\|_{\infty}+1)^{2-m}\|u_0\|_{L^1(\RN)}. }$  This is a bound independent of the parameters $\epsilon,\delta, R$ and $\mu$.

\medskip

\noindent$\bullet$\textbf{ Upper bound in the case $m\in (2,3)$}.
\begin{align*}
\int_{B_R}F_\mu(\widehat u_0)dx &= -\frac{1}{(m-2)(3-m)}\int_{B_R}[(\widehat u_0+\mu)^{3-m}-\mu^{3-m}]dx + \frac{1}{m-2}\mu^{2-m}\int_{B_R} \widehat u_0dx    \\
&\le \frac{1}{m-2}\mu^{2-m}\int_{B_R} \widehat u_0dx    \le \frac{1}{m-2}\mu^{2-m}\int_{\RN} u_0dx.
\end{align*}

This upper bound will allow us to obtain compactness arguments in $\epsilon$ and $R$ for fixed $\mu$.  We will be able to control $\int_{B_R}F_\mu(\widehat u_0)dx -\int_{B_R} F_\mu(U_1(t))dx  $ uniformly in $\mu$, after passing to the limit as $\epsilon\to 0$ and $R\to \infty$, due to a exponential decay result on the solution at time $t\in[0,T]$ that we will prove in Section \ref{SectExponTail} and the conservation of mass.

\noindent\textbf{Remark.} These techniques do not apply in the case $m\ge 3$ because even an exponential decay on the solution is not enough to control the terms in the first energy estimate.

\subsubsection{Second energy estimate} Similar computations to \eqref{SecondEnergy} yields to the following energy inequality
\[
\frac{1}{2}\frac{d}{dt}\int_{B_R}|\HH_s^\epsilon[U_1]|^2dx\leq - \delta \int_{B_R}|\nabla \HH_s^\epsilon[U_1]|^2dx-\int_{B_R}(U_1+\mu)^{m-1}|\nabla\KK_s^\epsilon[U_1]|^2dx.
\]
This implies that, for all $0<t<T$ we have
\begin{equation}\label{secondenergyEps}
\frac{1}{2}\int_{B_R}|\HH_s^\epsilon[U_1(t)]|^2dx+ \delta \int_0^T\int_{B_R}|\nabla \HH_s^\epsilon[U_1]|^2dx+\int_0^T\int_{B_R}(U_1+\mu)^{m-1}|\nabla\KK_s^\epsilon[U_1]|^2dx\leq \frac{1}{2}\int_{B_R}|\HH_s^\epsilon[\widehat{u}_0]|^2dx .
\end{equation}
Note that the last integral is well defined as long as $u_0 \in L^1(\RN)\cap L^\infty(\RN)$.

\section{Exponential tail control in the case $m\ge 2$ }\label{SectExponTail}

In this section and the next one we will give the proof of Theorem \ref{Thm2PMFP}. Weak solutions of the original problem are constructed by passing to the limit after a tail control step. We develop a comparison method with a suitable family of barrier functions, that in \cite{CaffVaz} received the name of \emph{true supersolutions}.

\begin{theorem}
Let $0<s<1/2$, $m\ge 2$ and let $U_1$ be the solution of Problem \eqref{model1Aprox}. We assume that $U_1$ is bounded $0 \le U_1(x,t)\le L$ and that $u_0$ lies below a function of the form
$$V_0(x) =Ae^{-a|x|}, \quad A,a >0.$$
If $A$ is large, then there is a constant $C>0$ that depends only on $(N,s,a,L,A)$ such that for any $T>0$ we will have the comparison
$$U_1(x,t) \le Ae^{Ct-a|x|} \quad \text{for all}\quad x  \in \RN, \ 0<t \le T.$$
\end{theorem}
\begin{proof}

\noindent $\bullet$ \textbf{Reduction.} By scaling we may put $a=L=1$. This is done by considering instead of $U_1$, the function $\tilde{U_1}$ defined as
\begin{equation}
U_1(x,t)=L \tilde{U_1}(ax,bt), \quad b=L^{m-1}a^{2-2s},
\end{equation}
which satisfies the equation
$$(\tilde{U_1})_t=\delta_1\Delta \tilde{U_1} + \nabla.(d_{\frac{\mu}{L}}(\tilde{U_1})\nabla \mathcal{K}_s^{\epsilon a}(\tilde{U_1})),$$
 with $\delta_1=a^{2s}\delta/L^{m-1}$. 
 Note that then $\tilde{U_1}(x,0)\le A_1e^{-|x|}$ with $A_1=A/L$. The corresponding bound for $\tilde{U_1}(x,t)$ will be  $\tilde{U_1}(x,t)\le A/L \ e^{C_1t-|x|}$ with $C_1=C/b=C\left(L^{m-1}a^{2-2s}\right)^{-1}$.

\noindent$\bullet$ \textbf{Contact analysis.} Therefore we assume that $0 \le U_1(x,0) \le 1$ and also that
$$U_1(x,0) \le Ae^{-r}, \quad r=|x|>0,$$
where $A>0$ is a constant that will be chosen below, say larger than $2$. Given constants $C, \epsilon$ and $\eta>0$, we consider a radially symmetric candidate for the upper barrier function of the form
$$\widehat{U}(x,t)=Ae^{Ct-r} + h A e^{\eta t},$$
and we take $h$ small.  Then $C$ will be determined in terms of $A$ to satisfy a true supersolution condition which is obtained by contradiction at the first point $(x_c,t_c)$ of possible contact of $u$ and $\widehat{U}.$

The equation satisfied by $u$ can be written in the form
\begin{equation}\label{eq1}
(U_1)_t=\delta \Delta U_1 + (m-1) (u+\mu)^{m-2} \nabla U_1 \cdot \nabla p + (U_1+\mu)^{m-1}\Delta p,\qquad p=\mathcal{K}_s^\epsilon[U_1].
\end{equation}
We will obtain necessary conditions in order for equation \eqref{eq1} to hold at the contact point $(x_c,t_c)$. Then, we prove there is a suitable choice of parameters $C,A,\eta,h, \mu$ such that the contact can not hold.

\noindent \textbf{Estimates on $u$ and $p$ at the first contact point.} For $0<s<1/2$, at the first contact point $(x_c,t_c)$ we have the estimates
$$
\partial_r U_1 =-A e^{Ct_c-r_c}, \quad \Delta U_1\le A e^{Ct_c-r_c}, \quad (U_1)_t \ge ACe^{Ct_c-r_c}+h \eta A e^{\eta t_c}.
$$Since we assumed our solution $u$ is bounded by $0\le u \le 1$, then
 \begin{equation}\label{est2}
 U_1(x_c,t_c)=Ae^{Ct_c-r_c}+h  A e^{\eta t_c} \le 1 .
 \end{equation}
Moreover, from \cite{CaffVaz} we have the following upper bounds for the pressure term at the contact point for $0<s<1/2$:
\begin{equation}\label{est1}
\Delta p (x_c,t_c) \le K_1, \quad (-\partial_r p) (x_c,t_c) \le K_2.
\end{equation}
Note that we are considering a regularized version of the $p$ used in $\cite{CaffVaz}$. Of course the estimates still true (maybe with slightly bigger constants) since $U_1$ is regular.

\noindent \textbf{Necessary conditions at the first contact point}. Equation \eqref{eq1} at the contact point $(x_c,t_c)$ with $r_c=|x_c|$, implies that
\begin{align*}
ACe^{Ct_c-r_c} + h \eta A e^{\eta t_c} &\le \delta A e^{Ct_c-r_c} +  (m-1) \left(U_1(x_c,t_c)+\mu \right)^{m-2} (-A e^{Ct_c-r_c}) (\partial_r p)+ \\ &+ (U_1(x_c,t_c)+\mu)^{m-1} \Delta p.
\end{align*}
We denote $\xi:=r_c + (\eta-C)t_c$. Using also \eqref{est1} with $K=\max\{K_1,K_2\}$, we obtain, after we simplify the previous inequality by $A e^{Ct_c-r_c}$,
\begin{align*}
C +h \eta  e^{\xi} &\le
\delta  + (m-1) \left(U_1(x_c,t_c)+\mu \right)^{m-2} K +(U_1(x_c,t_c)+\mu)^{m-2} (1+h e^{\xi} +\frac{\mu}{A}e^{r_c-Ct_c})K,
\end{align*}
and equivalently
\begin{align*}
C + \epsilon \eta  e^{\xi} &\le
\delta  + K\left(u(x_c,t_c)+\mu \right)^{m-2} \left(m+h e^{\xi} +\frac{\mu}{A}e^{r_c-Ct_c}\right).
\end{align*}
We take $C=\eta$ and $\displaystyle{\frac{\mu}{A} \le h}$. Then
\begin{align*}
C + h C e^{r_c} &\le \delta  + K\left(U_1(x_c,t_c)+\mu \right)^{m-2} \left(m+h e^{r_c} +h e^{r_c-Ct_c}\right).
\end{align*}
Moreover,
\begin{align*}
C + h C e^{r_c} &\le \delta  + K\left(U_1(x_c,t_c)+\mu \right)^{m-2} \left(m+2h e^{r_c}\right).
\end{align*}
By \eqref{est2} we have that
$$
\mu < U_1(x_c,t_c)+\mu <1+\mu.
$$
Since $m\ge 2$,  then
$$
C + h C e^{r_c} \le \delta  + K\left(1+\mu \right)^{m-2} \left(m+2h e^{r_c} \right) .
$$
This is impossible for $C$ large enough such that
\begin{equation}
 C \ge \delta+mK(1+\mu)^{m-2} \quad \text{and} \quad C \ge2 K\left(1+\mu \right)^{m-2}.
\end{equation}
Since $\mu<1$ and $\delta <1$, then we can choose $C$ as constant, only depending on $m$ and $K$.

\end{proof}

\begin{theorem}Let $1/2\le s < 1$, $m\ge 2$. Under the assumptions of the previous theorem,
the stated tail estimate works locally in time. The global statement must be
replaced by the following: there exists an increasing function $C(t)$ such that
\begin{equation}\label{EstUpperExp}
u(x,t) \le Ae^{C(t)t-a|x|} \quad \text{for all }x\in \RN \text{ and  all }0\le t\le T.
\end{equation}
\end{theorem}
\begin{proof}
The proof of this result is similar to the one in \cite{CaffVaz} but with a technical adaptation to our model.
When $N\ge 2$, $1/2\le s < 1$, the upper bound $\Delta p (x_c,t_c) \le K_0$ at the first contact point holds. Moreover, in \cite{CaffVaz}, the following upper bound for $ (-\partial_r p) (x_c,t_c) $ is obtained,
\[
(-\partial_r p) (x_c,t_c) \leq K_1+ K_2 ||U_1(t)||_1^{1/q}||U_1(t)||_\infty^{(q-1)/q},
\]
where $1\leq q < N/(2s-1)$. We know that $||U_1(t)||_\infty\leq 1$ and before the first contact point we have that $U_1(x,t)\leq Ae^{ct} e^{-|x|}$, hence $||U_1(t)||_1\leq K_3 A e^{Ct}$. Therefore, if we consider $K=\max\{K_0, K_1, K_2K_3^{1/q}\}$ we have that
\begin{equation}\label{est3}
\Delta p (x_c,t_c) \le K, \qquad (-\partial_r p) (x_c,t_c) \le K + K A^{1/q}e^{Ct_c/q}.
\end{equation}
Using this estimates in the equation we obtain
\begin{align*}
C + h \eta  e^{\xi} \le &  \ \delta  + K(m-1) \left(u(x_c,t_c)+\mu \right)^{m-2} (1 +  A^{1/q}e^{Ct_c/q}) + \\
&+K(u(x_c,t_c)+\mu)^{m-2} (1+he^{\xi} +\frac{\mu}{A}e^{r_c-Ct_c}).
\end{align*}
We put $C=\eta$, $h=\mu/A$  and use that $\mu < u(x_c,t_c)+\mu <1+\mu$ to get
\begin{align*}
C + hC e^{r_c} \le &  \ \delta  +KA^{1/q}(m-1) \left(1+\mu \right)^{m-2} e^{Ct_c/q} + K(1+\mu)^{m-2} \left( m+2h e^{r_c}\right).
\end{align*}
We consider $\mu<1$. The contradiction argument works as before with the big difference that we must restrict the time so that $e^{Ct_c/q} \le 2$, which happens if
$$
t_c \le T_1=(q\log 2) /C.
$$
Then
$$
C + h C e^{r_c} \le \delta +2^{m-1}KA^{1/q}(m-1) + 2^{m-2}Km +2^{m-2}Kh e^{r_c}.  $$
Since $A>1$ and $\delta<1$, and hence $2^{m-1}KA^{1/q}(m-1) + 2^{m-2}Km < 2^{m-1}KA^{1/q}(2m-1)$, we get a contradiction by choosing $C$ such that:
$$
C= 2^{m}KA^{1/q}m \geq \delta + 2^{m-1}KA^{1/q}(2m-1).
$$
We have proved that there will be no contact with the barrier
\[
B_1(x,t)=A e^{Ct-|x|}
\]
for $ t<T_1=c_1 A^{-1/q}$ where $c_1=\frac{q \log 2}{K m 2^m} $.

 We can repeat the argument for another time interval by considering the problem with initial value at time $T_1$, that is,
\[
U_1(x,T_1)\leq A e^{CT_1-|x|}=A_1 e^{-|x|} \mbox{ where } A_1=A e^{C T_1},
\]
and we get $U(x,t)\leq e^{C_1t-|x|}$ for $T_1\leq t<T_2=c_1 A^{-1/q} e^{-CT_1/q}$ where $C_1=C e^{CT_1/q}$. In this way we could find an upper bound to a certain time for the solution depending on the initial data through $A$.

When $N=1$, $1/2\le s < 1$, the operator $\partial_r p$ and $\Delta p$ are considered in the sense given in Section \ref{SubsectN1}.
\end{proof}

\section{Existence of weak solutions for $m\in (1,3)$}\label{limitssec}

\subsection{Limit as $\epsilon \to 0$}\label{SubsectEps}

We begin with the limit as $\epsilon \to 0$ in order to obtain a solution of the equation
\[(U_2)_t=\delta \Delta U_2 + \nabla \cdot (d_\mu(U_2)\nabla \KK_s[U_2]).
\tag{$P_{\delta\mu R}$}\label{ProblemMuDeltaR}
\]

Let $U_1$ be the solution of \eqref{model1Aprox}. We fix $\delta, \mu$ and $R$ and we argue for $\epsilon$ close to $0$.
Then, by the energy formula \eqref{energy2} and the estimates from Section \ref{SectApriori} we obtain that
\begin{equation}\label{energy3}
\delta \int_0^t \int_{B_R}\frac{|\nabla U_1|^2}{(U_1+\mu)^{m-1}}dx dt \le C(\mu,m,\widehat u_0), \quad
\int_0^t \int_{B_R}\left|\nabla \mathcal{H}_s^\epsilon[U_1]\right|^2 dx dt \le C(\mu,m, \widehat u_0),
\end{equation}
valid for all $\epsilon >0.$ Since $\|U_1\|_\infty \le \|u_0\|_{\infty}$ for all $\epsilon>0$, then
$$
\int_0^t \int_{B_R}|\nabla U_1|^2dx dt \le C(\mu,m,u_0) ( \|u_0\|_{\infty} + 1)^{m-1}, \quad \forall \epsilon>0.
$$
We recall that in the case $m\in (1,2)$ the constant $C$ is independent of $\mu$, that is $C=C(m, u_0)$.

\medskip

\noindent \textbf{I. Convergence as $\epsilon \to 0$. } We perform an analysis of the family of approximate solutions $(U_1)_\epsilon$ in order to derive a compactness property in suitable functional spaces.

\noindent $\bullet$  Uniform boundedness: $U_1  \in L^{\infty}(Q_{T,R})$, and the bound $||U_1(t)||_{L^\infty(\mathbb{R}^N)}\leq||u_0||_{L^\infty(\mathbb{R}^N)}$ is independent of $\epsilon,\delta,\mu$ and $ R$ for all $t>0$.  Moreover $||U_1(t)||_{L^1(\mathbb{R}^N)}\leq||u_0||_{L^1(\mathbb{R}^N)}$ for all $t>0$.

\noindent $\bullet$ Gradient estimates. From the energy formula \eqref{energy3} we derive
$$
 U_1\in L^2([0,T]: H_{0}^1(B_R)), \quad \nabla \mathcal{H}_s^\epsilon [U_1]\in L^2([0,T]: L^2(\RN))
$$
uniformly bounded for $\epsilon>0.$  Since $\nabla \mathcal{H}_s^\epsilon [U_1]$ is \lq\lq a derivative of order $1-s$ of $U_1$'', we conclude that
\begin{equation}\label{space1}U_1\in L^2([0,T],H^{1-s}(\RN)).
\end{equation}

\noindent $\bullet$ Estimates on the time derivative $(U_1)_t$: we use the equation \eqref{model1Aprox} to obtain that
\begin{equation}\label{space2}
(U_1)_t\in L^2([0,T]:H^{-1}(\RN))
\end{equation}

\noindent as follows:

\noindent (a) Since $U_1\in L^2([0,T]: H_{0}^1(B_R))$ we obtain that $\Delta U_1\in L^2([0,T]: H^{-1}(\RN))$.

\noindent (b) As a consequence of the Second Energy Estimate and the fact that $U_1 \in L^\infty(Q_T)$, we have that $d_\mu(U_1) \nabla \KK_s^\epsilon[U_1] \in L^2 ([0,T]: L^2(\RN)  )$, therefore $\nabla \cdot (d_\mu(U_1)\nabla \KK_s[U_1])) \in L^2([0,T]:H^{-1}(\RN))$.

Now, since $||(U_1)_t||_{L^1_t([0,T]: H^{-1+s}(\RN))}\leq T^{1/2}||(U_1)_t||_{L^2_t([0,T]: H^{-1+s}(\RN))} $, expressions \eqref{space1} and \eqref{space2}, allow us to apply the compactness criteria of Simon, see Lemma \ref{ConvSimon1} in the Appendix, in the context of
$$
H^{1-s}(\RN)\subset L^2(\RN) \subset H^{-1}(\RN),
$$
and we conclude that the family of approximate solutions $(U_1)$ is relatively compact in $L^2([0,T]:L^2(\RN))$. Therefore, there exists a limit $(U_1)_{\epsilon,\delta, \mu, R} \to (U_2)_{\delta, \mu, R}$ as $\epsilon \to 0$ in $L^2([0,T]:L^2(\RN)),$ up to subsequences. Note that, since $(U_1)_\epsilon$ is a family of positive functions defined on $B_R$ and extended to $0$ in $\RN\setminus B_R$, then the limit $U_2=0$ a.e. on $\RN\setminus B_R$.  We obtain that
\begin{equation}\label{conv1}
U_1\stackrel{\epsilon \to 0}{\longrightarrow}U_2   \quad \text{in }L^2([0,T]:L^2(B_R))=L^2(B_R\times [0,T]).
\end{equation}

\medskip

\noindent \textbf{II. The limit is a solution of the new problem \eqref{ProblemMuDeltaR}}. More exactly, we pass to the limit as $\epsilon \to 0$ in the definition \eqref{weaksolAprox} of a weak solution of Problem \eqref{model1Aprox} and prove that the limit $U_2(x,t)$ of the solutions $U_1 (x,t)$ is a solution of Problem \eqref{ProblemMuDeltaR}.
The convergence of the first integral in \eqref{weaksolAprox} is justified by \eqref{conv1} since
\[
\left|\int_0^T\int_{B_R} (U_1-U_2)(\phi_t-\delta \Delta \phi)dxdt\right|\leq ||U_1-U_2||_{L^2(B_R\times [0,T])}||\phi_t-\delta \Delta \phi||_{L^2(B_R\times [0,T])}.
\]

Convergence of the second integral in \eqref{weaksolAprox} is consequence of the second energy estimate \eqref{secondenergyEps} as we show now. First we note that
\[
||(U_1+\mu)^{\frac{m-1}{2}}\nabla \KK_s^\epsilon[U_1]||_{L^2(B_R\times(0,T))}\leq C
\]
for some constant $C>0$ independent of $\epsilon$. Then, Banach-Alaoglu ensures that there exists a subsequence such that
\[
(U_1+\mu)^{\frac{m-1}{2}}\nabla \KK_s^\epsilon[U_1]\, \stackrel{\epsilon \to 0}{\longrightarrow}\, v \mbox{ in } L^2(B_R\times(0,T)) \mbox{ weakly }.
\]
Moreover, it is trivial to show that $(U_1+\mu)^{-\frac{m-1}{2}}\, \stackrel{\epsilon \to 0}{\longrightarrow}\, (U_2+\mu)^{-\frac{m-1}{2}}$ in $L^2(B_R\times(0,T))$. Then
\[
\nabla \KK_s^\epsilon[U_1]=\frac{ (U_1+\mu)^{\frac{m-1}{2}}}{ (U_1+\mu)^{\frac{m-1}{2}}}\nabla \KK_s^\epsilon[U_1]
\, \stackrel{\epsilon \to 0}{\longrightarrow} \,
\frac{v}{ (U_2+\mu)^{\frac{m-1}{2}}} \mbox{ in } L^1(B_R\times(0,T)).
\]
In particular we get that there exists a limit of $\nabla \KK_s^\epsilon[U_1]$ as $\epsilon \to 0$ in any $L^p(B_R\times (0,T) )$ with $1\leq p\leq \infty$. Now we need to identify this limit. The following Lemma shows that $\nabla \KK_s^\epsilon[U_1]\stackrel{\epsilon \to 0}{\longrightarrow}\nabla \KK_s[U_2]$ in distributions, and so we can conclude convergence in $L^2(B_R\times(0,T))$.

\begin{lemma}\label{lemmaConvK}
Let $s\in(0,1)$ ($0<s<1/2$ if $N=1$). Then

(1) $\KK_s^\epsilon[U_1]  \stackrel{\epsilon \to 0}{\longrightarrow} \KK_s[U_2]$ in $L^1(B_R\times (0,T))$.

\noindent (2) $\displaystyle \int_0^T \int_{B_R}  \KK_s^\epsilon[U_1]  \nabla  \psi \, dx dt  \stackrel{\epsilon \to 0}{\longrightarrow}  \int_0^T \int_{B_R} \KK_s[U_2] \nabla \psi \, dx dt  $ for every $\psi \in C_c^\infty(Q_T).$
\end{lemma}

\begin{proof}
For the first part of the Lemma, we split the integral as follows,
\[
\int_0^T\int_{B_R} (\KK_s^\epsilon[U_1] - \KK_s[U_2])dxdt=\int_0^T\int_{B_R} (\KK_s^\epsilon[U_1] - \KK_s[U_1])dxdt+\int_0^T\int_{B_R} (\KK_s[U_1] - \KK_s[U_2])dxdt.
\]
Note that $\KK_s[U_1]=K_s*U_1$ with $K_s\in L^1_{loc}(\R^N)$ and $\KK_s^\epsilon[U_1]=K_s^\epsilon*U_1$ with $K_s^\epsilon=\rho_\epsilon*K_s$ where $\rho_\epsilon$ is a standard mollifier.
Then the first integral on the right hand side goes to zero as $\epsilon\to 0$.  The second integral goes to zero  with $\epsilon$ as consequence of \eqref{conv1}.

The second part of the Lemma is just a corollary of the first part.
\[
\left|\int_0^T \int_{B_R} (\KK_s^\epsilon[U_1]-\KK_s[U_2])\nabla \psi dx dt  \right|\leq ||\nabla \psi||_\infty ||\KK_s^\epsilon[U_1]-\KK_s[U_2]||_{L^1(B_R\times (0,T))}.
\]
\end{proof}
The remaining case $N=1$, $s\in (1/2,1)$ will be explained in Section \ref{SubsectN1}.
We conclude that,
$$
\int_0^T\int_{B_R} d_\mu(U_1) \nabla \KK_s^\epsilon [U_1]\nabla \phi dxdt \to \int_0^T\int_{B_R} d_\mu(U_2) \nabla \KK_s[U_2]\nabla \phi dxdt, \quad \text{as }\epsilon \to 0.
$$

Note that we can obtain also that $\nabla \HH_s^\epsilon[U_1]\stackrel{\epsilon \to 0}{\longrightarrow}\nabla \HH_s[U_2]$ in $L^2(B_R\times(0,T))$ using the same argument. This allows us to pass to the limit in the energy estimates.

The conclusion of this step is that we have obtained a weak solution of the initial value problem \eqref{ProblemMuDeltaR} posed in $B_R\times [0,T]$ with homogeneous Dirichlet boundary conditions. The regularity of $U_2$, $\HH_s[U_2]$ and $\KK_s[U_2]$ is as stated before. We also have the energy formulas
\begin{equation}\label{energyidentityR}
\int_{B_R} F_\mu(U_2(t))dx+\delta\int_0^t \int_{B_R} \frac{|\nabla U_2|^2}{d_\mu(U_2)}dxdt+\int_0^t  \int_{B_R} |\nabla \HH_s[U_2]|^2dxdt=\int_{B_R}F_\mu(u_0)dx.
\end{equation}

\begin{align*}
\frac{1}{2}\int_{B_R}|\HH_s[U_2(t)]|^2dx+& \delta \int_0^t\int_{B_R}|\nabla \HH_s[U_2]|^2dx\, dt+\int_0^t\int_{B_R}(U_2+\mu)^{m-1}|\nabla\KK_s[U_2]|^2dx \,dt\\
&\leq \frac{1}{2}\int_{B_R}|\HH_s[\widehat{u}_0]|^2dx . \nonumber
\end{align*}

We do not pass now to the limit as $\delta\to0$, because we lose $H^1$ estimates for $U_2$ and we deal with the problem caused by the boundary data. Therefore, we keep the term $\delta \Delta U_2$.

\subsection{Limit as $R\to \infty$}

We will now pass to the limit as $R\to \infty$. The estimates used in the limit on $\epsilon$ in Section \ref{SectApriori} are also independent on $R$. Then the same technique may be applied here in order to pass to the limit as $R\to \infty$. Indeed, we get that $\displaystyle U_3=\lim_{R\to \infty}U_2$ in $L^2(\R^N\times (0,T))$ is a weak solution of the problem in the whole space
\[
(U_3)_t=\delta \Delta U_3 + \nabla\cdot\left( (U_3+\mu)^{m-1}\nabla \KK_s[U_3]\right) \ \ \ x\in \RN, \ \ t>0.
\tag{$P_{\mu\delta}$}\label{ProblemMuDelta}\]

This problem satisfies the property of conservation of mass, that we prove next.

\begin{lemma}
Let $u_0\in L^1(\RN)\cap L^\infty(\RN)$. Then the constructed non-negative solution of Problem \eqref{ProblemMuDelta} satisfies
\begin{equation}\label{conservationofmass}
\int_{\RN} U_3(x,t)dx=\int_{\RN} u_0(x)dx \mbox{ for all } t>0.
\end{equation}
\end{lemma}
\begin{proof}
Let $\varphi \in C^\infty_0(\RN)$ a cutoff test function  supported in the ball $B_{2R}$ and such that $\varphi\equiv 1$ for $|x|\leq R$, we recall the construction in the Appendix \ref{cutoffSect}. We get
\[
\int_{B_{2R}}(U_3)_t \varphi dx=\delta \int_{B_{2R}} U_3\Delta \varphi dx - \int_{B_{2R}} (U_3+\mu)^{m-1}\nabla \KK_s[U_3] \cdot \nabla \varphi dx=I_1+I_2.
\]
Since $U_3(t)\in L^1(\RN)$ for any $t \geq0$, we estimate the first integral as $I_1=O(R^{-2})$ and then $I_1\to 0$ as $R\to \infty$. For the second integral we have
\[
I_2=\int_{B_{2R}}\KK_s[U_3]\nabla\cdot\left((U_3+\mu)^{m-1}\nabla \varphi\right)dx,
\]
\[
I_2=(m-1)\int_{B_{2R}}\KK_s[U_3](U_3+\mu)^{m-2}\nabla u\cdot \nabla \varphi dx+\int_{B_{2R}}\KK_s[U_3](U_3+\mu)^{m-1}\Delta \varphi dx=I_{21}+I_{22}.
\]
Since $\nabla U_3 \in L^2(\RN)$ and $U_3\in L^\infty(\RN)$,
\[
|I_{21}|\leq C ||(U_3+\mu)^{m-2}||_\infty\left(\int_{B_{2R}}|\nabla U_3|^2 dx\right)^{1/2}\left(\int_{B_{2R}}|\KK_s[U_3]|^2|\nabla \varphi|^2 dx\right)^{1/2} .
\]
Now $\nabla \varphi =O(R^{-1})$, $\nabla\varphi \in L^p$ with $p>N$, so we need $\KK_s[U_3] \in L^q$ for $\displaystyle q<2\frac{1}{1-\frac{1}{N/2}}=\frac{2N}{N-2}$ which is true since $\KK[U_3]\in L^{q}$ for $q>q_0=N/(N-2s)$, and $q_0<2N/(N-2)$ if $4s<N+2$. So, since $p>N$,
\begin{eqnarray*}
|I_{21}|&\leq& C  \left(\int_{B_{2R}}|\nabla \KK_s[U_3]|^q dx\right)^{1/q}\left(\int_{B_{2R}}|\nabla \varphi|^p dx\right)^{1/p} \\
&\leq& C\left(\int_{B_{2R}}R^{-p} dx\right)^{1/p}\leq C R^{\frac{N-p}{p}}\stackrel{R\to \infty}{\lra} 0.
\end{eqnarray*}
For $I_{22}$, we will use the same trick of the previous section,
\[
I_{22}=\int_{B_{2R}}\KK_s[U_3]\left[(U_3+\mu)^{m-1}-\mu^{m-1}\right]\Delta \varphi dx+\mu^{m-1}\int_{B_{2R}}\KK_s[U_3]\Delta \varphi dx=I_{221}+I_{222}.
\]
Now,
\[
I_{222}=\mu^{m-1}\int_{B_{2R}}U_3\KK_s[\Delta \varphi ]dx=\mu^{m-1}||U_3||_1 O(R^{-2+2s})\stackrel{R\to\infty}{\lra}0,
\]
where we use the fact that $\KK\Delta$ has homogeneity $2-2s>0$ as a differential operator. Also,
\[
I_{221}=\int_{B_{2R}}f'(\xi)U_3\KK_s(U_3)\Delta \varphi dx ,
\]
where $f(s)=s^{m-1}$ and $\xi \in [\mu,\mu+U_3(x)]$. Again, since $U_3\in L^\infty$,  there exists a global bound for $f'(\xi)$, that is, $f'(\xi)\leq (m-1)\max\{\mu^{m-2}, (\mu+||U_3||_\infty)^{m-2}\}$ and so integral $I_{221}\to 0$ as $R\to \infty$ (details could be found in \cite{CaffVaz}).

In the limit $R\to \infty$, $\varphi\equiv 1$ and we get \eqref{conservationofmass}.
\end{proof}

\noindent \textbf{Consequence. }The estimates done in Section \ref{SectApriori} can be improved  passing to the limit $R\to \infty$, since the conservation of mass \eqref{conservationofmass} eliminates some of the integrals that presented difficulties when trying to obtain upper bounds independent of $\mu$. Therefore, we compute the following terms in the energy estimate \eqref{energyidentityR}.

\noindent For $m\not=2,3$ we have
\begin{align}\label{RinftyL1}
&\int_{B_R}F_\mu(u_0)dx - \int_{B_R}F_\mu(U_2)dx= \nonumber\\
&= C\int_{B_R} [(u_0+\mu)^{3-m}-\mu^{3-m}]dx -\frac{1}{2-m}\mu^{2-m}\int_{B_R}u_0 dx \nonumber\\
& -  C\int_{B_R} [(U_2+\mu)^{3-m}-\mu^{3-m}]dx +\frac{1}{2-m}\mu^{2-m}\int_{B_R}U_2 dx \nonumber \\
&\longrightarrow\quad C \int_{\RN} [(u_0+\mu)^{3-m}-\mu^{3-m}]dx -  C \int_{\RN} [(U_3+\mu)^{3-m}-\mu^{3-m}]dx,
\end{align}
as $R \to \infty.$ We use the notation $C=\frac{1}{(2-m)(3-m)}.$

\noindent For $m=3$ we have
\begin{align}\label{RinftyL2}
&\int_{B_R}F_\mu(u_0)dx - \int_{B_R}F_\mu(U_2)dx= \nonumber\\
&= - \int_{B_R} \log\left(1+\frac{u_0}{\mu}  \right)dx  + \frac{1}{\mu}\int_{B_R} u_0 dx + \int_{B_R} \log\left(1+\frac{U_2}{\mu} \right)dx  - \frac{1}{\mu}\int_{B_R} U_2 dx \nonumber\\
&\longrightarrow \quad \int_{\RN} \log\left(1+\frac{U_3}{\mu}\right)dx - \int_{\RN} \log\left(1+\frac{u_0}{\mu} \right)dx \quad \text{as }R \to \infty.
\end{align}

\noindent The following theorem summarizes the results obtained until now.

\begin{theorem}\label{ThmExist2}
Let $m>1$ and $u_0\in L^1(\RN)\cap L^\infty (\RN)$ be non-negative. Then there exists a weak solution $U_3$ of Problem \eqref{ProblemMuDelta} posed in $\RN\times(0,T)$ with initial data $u_0$. Moreover, $U_3\in L^\infty([0,\infty): L^1(\RN))$,  and for all $t>0$ we have
\[\int_{\RN}U_3(x,t)dx=\int_{\RN}u_0(x)dx\]
and $||U_3(\cdot,t)||_\infty\leq ||u_0||_\infty$. The following energy estimates also hold:

\textbf{(i) First energy estimate:}\\
$\bullet$ If $m=3$,
\begin{eqnarray}\label{energy4}
\delta\int_0^t \int_{\RN} \frac{|\nabla U_3|^2}{(U_3+\mu)^2}dxdt&+&\int_0^t  \int_{\RN} |\nabla \HH_s[U_3]|^2dxdt
+\int_{\RN}\log\left(\frac{u_0}{\mu}+1\right)dx   \\
&\le&\int_{\RN}\log\left(\frac{U_3(t)}{\mu}+1\right)dx.\nonumber
\end{eqnarray}
$\bullet$ If $m\not=2,3$
and
\begin{eqnarray}\label{energy5}
\delta\int_0^t \int_{\RN} \frac{|\nabla U_3|^2}{(U_3+\mu)^{m-1}}dxdt&+&\int_0^t  \int_{\RN} |\nabla \HH_s[U_3]|^2dxdt +  \\
+C\int_{\RN}\left[(U_3(t)+\mu)^{3-m}-\mu^{3-m}\right]dx &\le &C\int_{\RN}\left[(u_0+\mu)^{3-m}-\mu^{3-m}\right]dx \nonumber
\end{eqnarray}
  where $C=C(m)=\frac{1}{(2-m)(3-m)}$.

  \textbf{(ii) Second energy estimate:}
\begin{align*}
\frac{1}{2}\int_{\RN}|\HH_s[U_3(T)]|^2dx+& \delta \int_0^T\int_{\RN}|\nabla \HH_s[U_3]|^2dx\, dt+\int_0^T\int_{\RN}(U_3+\mu)^{m-1}|\nabla\KK_s[U_3]|^2dx \,dt\\
&\leq \frac{1}{2}\int_{\RN}|\HH_s[u_0]|^2dx . \nonumber
\end{align*}
\end{theorem}

\subsection{Limit as $\mu\to 0$}

Similarly to the previous limits we can prove that $\displaystyle U_4=\lim_{\mu \to 0}U_3$ in $L^2(\R^N\times (0,T))$ when $m\in (1,3)$. Then $U_4$ will be a solution of problem
\[
(U_4)_t=\delta \Delta U_4 + \nabla\cdot\left( U_4^{m-1}\nabla \KK_s[U_4]\right) \ \ \ x\in \RN, \ \ t>0.
\tag{$P_{\delta}$}\label{ProblemDelta}\]
In order to pass to the limit, we need to find uniform bounds on $\mu>0$ for terms 3 and 4 of the energy estimates \eqref{energy4} and \eqref{energy5}.

\noindent\textbf{Uniform upper bounds }

\noindent$\bullet$ Case $m \in (1,2)$. By the Mean Value Theorem,
\begin{align*}
\frac{1}{(m-2)(3-m)}\int_{\RN}&\left[(u_0+\mu)^{3-m}-\mu^{3-m}\right]dx\leq \frac{1}{(m-2)}\int_{\RN} (u_0+\mu)^{2-m}u_0dx\\ &\leq \frac{(||u_0||_\infty+1)^{2-m}}{m-2}\int_{\RN} u_0dx.
\end{align*}
This bound is independent of $\mu$.

\noindent$\bullet$ Case $m \in (2,3)$. The function $f(\zeta)=\zeta^{3-m}$ is concave and so $f(U_3+\mu)\leq f(\mu)+f(U_3)$. In this way,
\[
\frac{1}{(2-m)(3-m)}\int_{\RN}\left[(U_3(t)+\mu)^{3-m}-\mu^{3-m}\right]dx \leq \frac{1}{(2-m)(3-m)}\int_{\RN} U_3(t)^{3-m}dx.
\]
The last integral is finite due to the exponential decay for $U_3$ that we proved in Section \ref{SectExponTail}. In this way, the last estimate is uniform in $\mu$.

\noindent\textbf{The limit is a solution of the new problem \eqref{ProblemDelta}}.
The argument from Section \ref{SubsectEps} does not apply for the limit
\begin{equation}\label{limitmu}
\int_0^T \int_{\R^N}(U_3+\mu)^{m-1}\nabla \KK_s[U_3]\nabla \phi dx dt \stackrel{\mu \to0}{\longrightarrow}\int_0^T \int_{\R^N}U_4^{m-1}\nabla \KK_s[U_4]\nabla \phi dx dt.
\end{equation}
In order to show that this convergence holds, we note that from the first energy estimate we get that
\[\nabla \HH_s [U_3] \in L^2((0,T): L^2(\R^N))\]
uniformly on $\mu$. Then $\nabla \KK_s[U_3] = \HH_s [\nabla \HH_s[U_3]]\in L^2((0,T): H^s(\R^N))$. Since for any bounded domain $\Omega$, $H^s(\Omega)$ is compactly embedded in $L^2(\Omega)$  then $\nabla \KK_s[U_3] \to \nabla \KK_s[U_4]$ as $\mu \to 0$ in $L^2(\Omega)$. Then we have the convergence \eqref{limitmu} since $U_3 \in L^\infty(\R^N)$ and $\phi$ is compactly supported.

\noindent\textbf{Remarks.} $\bullet$ In the case $m=2$ the corresponding term is $\intr U_3 \log^-(U_3+\mu)dx$ which is uniformly bounded if $U_3$ has an exponential tail. This has been proved by Caffarelli and V\'azquez in \cite{CaffVaz}. We do not repeat the proof here.

\noindent$\bullet$ The case $m\ge 3$ is more difficult since we can not find uniform estimates in $\mu>0$ for the energy estimates that allow us to pass to the limit.

\subsection{Limit as $\delta\to 0$}

We will prove that there exists a limit $u = \lim_{\delta \to 0}U_4 $ in $L^2(\R^N\times (0,T))$
and that $u(x,t)$ is a weak solution to Problem \eqref{model1}. Thus, we conclude the proof of Theorem \ref{Thm1PMFP} stated in the introduction of this chapter.

We comment on the differences that appear in this case.
From the first energy estimate we have that
\[
 \delta\int_0^T\int_{\RN}\frac{|\nabla U_4|^2}{U_4^{m-1}}dxdt\leq C(m,u_0),
\]
which gives us that $\delta \nabla U_4 \in L^2(Q_T)$ since $U_4\in L^\infty(Q_T)$. Then, as in Section \ref{SubsectEps}, we have that $\delta \Delta U_4 \in H^{-1}(\RN)$ uniformly in $\delta$.  Also $\nabla(U_4^{m-1}\nabla \KK_s[U_4])\in H^{-1}(\RN) $ as before.  Then $(U_4)_t \in H^{-1}(\RN)$ independently on $\delta$. Therefore we use the compactness argument of Simon
to obtain that there exists a limit
\[ U_4(x,t) \to u(x,t) \quad L^2((0,T)\times \RN).
\]
Now we show that $u$ is the weak solution of Problem \eqref{model1}.  It is trivial that $\delta\int_0^T\int_{\R^N} U_4 \Delta \phi \to0$ as $\delta \to 0$. On the other hand, $\nabla \KK_s[U_4]=\HH_s[\nabla \HH_s[U_4]]\in L^2_{loc}(Q_T)$ uniformly on $\delta>0$ since $\nabla \HH_s[U_4]\in L^2(Q_T)$ uniformly on $\delta>0$. In this way, $\nabla \KK_s[U_4]$ has a weak limit in $L^2_{loc}(Q_T)$. As in Lemma \ref{lemmaConvK} (2) we can identify this limit and so on,  $\nabla \KK_s[U_4]\to \nabla \KK_s[u] $ weakly  in $L^2_{loc}(Q_T)$ as $\delta \to 0$ and therefore
\[
\int_0^T \int_{\R^N}U_4^{m-1}\nabla \KK_s[U_4]\nabla \phi dx dt \stackrel{\delta \to0}{\longrightarrow}\int_0^T \int_{\R^N}u^{m-1}\nabla \KK_s[u]\nabla \phi dx dt.
\]
since $U_4^{m-1}\to u^{m-1}$ in $L^2_{loc}(Q_T)$ as $\delta \to0$.

\subsection{Dealing with the case $N=1$ and $1/2<s<1$}\label{SubsectN1}
As we have commented before, the operator $\KK_s$ is not well defined when $N=1$ and $1/2<s<1$ since the kernel $|x|^{1-2s}$ does not decay at infinity, indeed it grows.  It makes no sense to think of equation \eqref{casen1} in terms of a pressure as before. This is maybe not very convenient, but it is not an essential problem, since equation \eqref{model1} can be considered in the following sense:
\begin{equation}\label{casen1}
  u_{t}(t,x) = \nabla \cdot \left(u^{m-1} (\nabla \KK_s)[u]\right)  \text{for } x \in \RN, \, t>0,
\end{equation}
where the combined operator $(\nabla \KK_s)$ is defined as the convolution operator
\[
(\nabla \KK_s)[u]:=(\nabla K_s)*u \quad \text{ with } \quad K_s(x)=\frac{c_s}{|x|^{1-2s}}.
\]
Other authors that dealt with $N=1$ have considered  operator $(\nabla \KK_s)$ before. They use the notation $\nabla^{2s-1}$ to refer to it. Note that
\[
\nabla K_s(x) = (-1+2s)c_s\frac{x}{|x|^{3-2s}},
\]
and so,  $\nabla K_s\in L^1_{loc}(\R)$ for $N=1$ and $1/2<s<1$. Moreover, $(\nabla \KK_s)$ is an integral operator in this range. As in Subsection \ref{aproxinvlap}, the operator $(\nabla \KK_s)$ is approximated by $(\nabla \KK_s)^{\epsilon}$ defined as
\[
(\nabla \KK_s)^{\epsilon}[u]=(\nabla K_s)^\epsilon*u \quad \text{ where } \quad  (\nabla K_s)^\epsilon= \rho_\epsilon*(\nabla K_s).
\]
Note that
\begin{equation}\label{conveps2}
(\nabla K_s)^\epsilon \stackrel{\epsilon \to 0}{\longrightarrow} \nabla K_s\quad  \text{in} \quad L^1_{loc}(\R),
\end{equation}
since $\nabla K_s\in L^1_{loc}(\R)$. It is still true that
\[
(\nabla \KK_s)[u]=\HH_s\left[ \nabla \HH_s [u]\right],
\]
since the operator $\HH_s$ is well defined for any $s\in(0,1)$ even in dimension $N=1$.

In this way, almost all the arguments from Section \ref{limitssec} apply by replacing $\nabla (\KK_s(u))$ for $(\nabla \KK_s)(u)$. The only exception is Lemma \ref{lemmaConvK} where the weak $L^2(\R)$ limit of $\nabla K_s [U_1]$ is identified. This argument is replaced by the following Lemma:
\begin{lemma} Let N=1 and $1/2<s<1$. Then
\[
\displaystyle \int_0^T \int_{B_R}  U_1  (\nabla \KK_s)^\epsilon [\psi] \, dx dt  \stackrel{\epsilon \to 0}{\longrightarrow}  \int_0^T \int_{B_R} U_2 (\nabla\KK_s)[ \psi] \, dx dt \qquad \forall \psi \in C_c^\infty(Q_T).
\]
\end{lemma}
\begin{proof}
\begin{equation*}
\begin{split}
\int_0^T \int_{B_R}  U_1  (\nabla \KK_s)^\epsilon [\psi] -U_2 (\nabla\KK_s)[ \psi] \, dx dt&=\int_0^T \int_{B_R}  (U_1 -U_2) (\nabla \KK_s)^\epsilon [\psi]  \, dx dt\\
&+\int_0^T \int_{B_R}  U_2 \big( (\nabla \KK_s)^\epsilon [\psi]-(\nabla\KK_s)[ \psi]\big) \, dx dt.
\end{split}
\end{equation*}
The first integral on the right hand side goes to zero with $\epsilon$ since $||(\nabla \KK_s)^\epsilon [\psi]||_{L^\infty (\R)}\leq K$ for some positive constant $K$ which does not depend on $\epsilon$ and $U_1\to U_2$ as $\epsilon \to 0$ in $L^2(B_R\times (0,T))$. The second integral also goes to zero as consequence of \eqref{conveps2} and the fact that $U_2 \in L^\infty(\R)$ uniformly on $\epsilon$.
\end{proof}

\section{Finite propagation property for $m\in [2,3)$}\label{SectionFiniteProp}

In this section we will prove that compactly supported initial data $u_0(x)$ determine the solutions $u(x,t)$ that have the same property for all positive times.

\begin{theorem}\label{ThmBarrierParabola}
Let $m\ge 2$. Assume $u$ is a bounded solution, $0\le u\le L$, of Equation \eqref{model1} with $\mathcal{K}=(-\Delta)^{-s}$ with $0<s<1$ ($0<s<1/2$ if $N=1$), as constructed in Theorem \ref{Thm2PMFP}. Assume that $u_0$ has compact support. Then $u(\cdot,t)$ is compactly supported for all $t>0$. More precisely, if $0<s<1/2$ and $u_0$ is below the ''parabola-like'' function
$$U_0(x)=a(|x|-b)^2,$$
for some $a,b>0$, with support in the ball $B_b(0)$, then there is a constant $C$ large enough, such that
$$
u(x,t) \le a(Ct-(|x|-b))^2.
$$
Actually, we can take $\displaystyle{C(L,a)=C(1,1)L^{m-\frac{3}{2}+s}a^{\frac{1}{2} -s}}$. For $1/2\le s <1$ a similar conclusion is true, but $C=C(t)$ is an increasing function of $t$ and we do not obtain a scaling dependence of $L$ and $a$.
\end{theorem}
\begin{proof}

The method is similar to the tail control section. We assume $u(x,t)\ge 0$ has bounded initial data $u_0(x)=u(x,t_0) \le L$, and also that $u_0$ is below the parabola $U_0(x)=a(|x|-b)^2$, $a,b>0.$ Moreover the support of $U_0$ is the ball of radius $b$ and the graphs of $u_0$ and $U_0$ are strictly separated in that ball. We take as comparison function $U(x,t)=a(Ct-(|x|-b))^2$ and argue at the first point in space and time where $u(x,t)$ touches $U$ from below. The fact that such a first contact point happens for $t>0$ and $x\ne \infty$ is justified by regularization, as before. We put $r=|x|.$

By scaling we may put $a=L=1$. We denote by $(x_c,t_c)$ this contact point where we have $u(x_c,t_c)=U(x_c,t_c)=(b+Ct_c-|x_c|)^2 .$ The contact can not be at the vanishing point $|x_f(t_c)|:=b+Ct_c$ of the barrier and this will be proved in  Lemma \ref{LemmaNoContactBdry}. We consider that $x_c$ lies at a distance $h>0$ from $|x_f(t_c)|=b+Ct_c$ (the boundary of the support of the parabola $U(x,t)$ at time $t_c$), that is
$$b+Ct_c-|x_c|=h>0.$$ Note that since $u \le 1$ we must have $|h|\le 1$. Assuming that $u$ is also $C^2$ smooth, since we deal with a first contact point $(x_c,t_c)$, we have that $u=U$, $\nabla (u-U)=0$, $\Delta(u-U) \le 0$, $(u-U)_t \ge 0,$ that is
$$
u(x_c,t_c)=h^2,\quad u_r=-2h, \quad \Delta u \le 2N, \quad u_t\ge 2Ch.
$$
For $p=\mathcal{K}_s(u)$ and using the equation $u_t=(m-1)u^{m-2}\nabla u \cdot \nabla p + u^{m-1}\Delta p$, we get the inequality
\begin{equation}\label{ineq1PMFP}
2Ch\le 2(m-1)h^{2m-3}\left(-\overline{p_r} + \frac{h}{2}\overline{\Delta p}\right) ,
\end{equation}
where $\overline{p_r}$ and $\overline{\Delta p} $ are the values of $p_r$ and $\Delta p$ at the point $(x_c,t_c)$.
In order to  get a contradiction, we will use estimates for the values of $\overline{p_r}$ and $\overline{\Delta p} $ already proved in \cite{CaffVaz} (see Theorem 5.1. of \cite{CaffVaz})
\begin{equation}\label{ineq2}
-\overline{p_r} \le K_1 + K_2 h^{1+2s} +K_3 h, \quad \overline{\Delta p} \le K_4.
\end{equation}
Therefore, inequality \eqref{ineq1PMFP} combined with the estimates \eqref{ineq2} implies that
\begin{equation}\label{ineq5}
2C\le 2(m-1)h^{2m-4}\left( K_1 + K_2h^{1+2s} + Kh\right),
\end{equation}
which is impossible for $C$ large (independent of $h$), since $m>2$ and $|h|\le 1$.  Therefore, there cannot be a contact point with $h\ne 0$. In this way we get a minimal constant $C=C(N,s)$ for which such contact does not take place.

\noindent Remark: For $m<2$, we do not obtain a contradiction in the estimate \eqref{ineq5}, since the term $K_1h^{2m-4}$ can be very large for small values of $|h|$.

\medskip

\noindent$\bullet$ \textbf{Reduction. Dependence on $L$ and $a$. }The equation is invariant under the scaling
\begin{equation}\label{scaling}
\widehat{u}(x,t)=Au(Bx,Tt)
\end{equation}
with parameters $A,B,T>0$ such that $T=A^{m-1}B^{2-2s}$.

\noindent {\sc Step I.} We prove that if $u$ has height $0\le u(x,t)\le 1$ and initially satisfies $u(x,0)=u_0(x)\le (|x|-b)^2$ then $u(x,t) \le U(x,t)=(Ct-(|x|-b))^2$ for all $t>0$.

\noindent {\sc Step II.} We search for parameters $A,\ B,\ T$ for which the function $\widehat{u}$ is defined by \eqref{scaling} satisfies
 $$0\le \widehat u(x,t)\le L,\quad \widehat u(x,0)\le \widehat{a}(|x|-\widehat{b})^2.$$
 An easy computation gives us
 $$
 A=L, \quad AB^2=\widehat{a}, \quad \widehat{b}=b/B.
$$
Moreover, by the relation between $A,B$ and $T$ we obtain $A=L$, $B=(\widehat{a}/L)^{1/2}$ and then $T=L^{m-2+s}\widehat{a}^{1-s}.$
Then $\widehat{u}(x,t)$ is below the upper barrier $\widehat{U}(x,t)=\widehat{a}(\widehat{C}t-(|x|-\widehat{b}))^2$ where the new speed is given by
$$
\widehat{C}=CA^{m-1}B^{1-2s}=CL^{m-\frac{3}{2}+s}\ \widehat{a}^{\frac{1}{2} -s}.
$$

\noindent$\bullet$ \textbf{Case $1/2\le s <1$.} The proof relies on estimating the term $\partial_r p$ at a possible contact point. This is independent on $m$ and it was done in \cite{CaffVaz}.
\end{proof}

\begin{lemma}\label{LemmaNoContactBdry}
Under the assumptions of Theorem \ref{ThmBarrierParabola} there is no contact between $u(x,t)$ and the parabola $U(x,t)$, in the sense that strict separation of $u$ and $U$ holds for all $t>0$ if $C$ is large enough.
\end{lemma}
\begin{proof}
We want to eliminate the possible contact of the supports at the lower part of the parabola, that is the minimum $|x|=Ct+b$. Instead of analyzing the possible contact point, we proceed by a change in the test function that we replace by
\begin{equation*}
U_\epsilon(x,t)=
\left\{
\begin{array}{ll}
 (Ct-(|x|-b))^2 + \epsilon(1+Dt)&\text{for }|x|\le b+Ct,\\[2mm]
\epsilon(1+Dt),&\text{for }|x|\ge b+Ct.
\end{array}
\right.
\end{equation*} The function $U_\epsilon$ is constructed from the parabola $U$ by a vertical translation $\epsilon(1+Dt)$ and a lower truncation with $1+Dt$ outside the ball $\{|x|\le b+Ct\}$.
Here $0<\epsilon<1$ is a small constant and $D>0$ will be suitable chosen.

We assume that the solution $u(x,t)$ starts as $u(x,0)=u_0(x)$ and touches for the first time the parabola $U_\epsilon$ at $t=t_c$ and spatial coordinate $x_c$.  The contact point can not be a ball $\{|x|\le b+Ct\}$ since $U_\epsilon$ is a parabola here and this case was eliminated in the previous Theorem \ref{ThmBarrierParabola}.
Consider now the case when the first contact point between $u(x,t)$ and $U_\epsilon(x,t)$ is when $|x_c|\ge b+Ct_c$. At the contact point we have that $u=U_\epsilon$, $\nabla (u-U_\epsilon)=0$, $\Delta(u-U_\epsilon) \le 0$, $(u-U_\epsilon)_t \ge 0.$ In this region the spatial derivatives of $U_\epsilon$ are zero, hence the equation gives us
$$
D\epsilon =(\epsilon(1+Dt_c))^{m-1} \overline{\Delta p},
$$
where $\overline{\Delta p}$ is the value of $\Delta p=(-\Delta )^{1-s}u $ at the point $(x_c,t_c)$.
 Since $\epsilon$ is small we get that the bound $u(x,t)\le U_1(x,t)$ is true for all $|x|\le \RN$.  This allows us to prove that that $\overline{\Delta p}$ is bounded by a constant $K$.  We obtain that $ D\epsilon \le (\epsilon(1+Dt_c))^{m-1}K.$ Since $m\ge 2$ and $\epsilon <1$, this implies that $$
D \le (1+Dt_c)^{m-1}K.
$$
We obtain a contradiction for large $D$, for example $D=2K$, and for
$$t_c < T_c=\frac{1}{2K} \left(2^{1/(m-1)}-1\right).$$
Therefore, we proved that a contact point between $u$ and $U_\epsilon$ is not posible for $t< T_c$, and thus $u(x,t) \le U_\epsilon(x,t)$ for $t<T_c$. The estimate on $t_c$ is uniform in $\epsilon$ and we obtain in the limit $\epsilon \to 0$ that
$$u(x,t) \le U(x,t)=(Ct-(|x|-b))\quad \text{for }t<\frac{1}{2K} \left(2^{1/(m-1)}-1\right).$$
As a consequence, the support of $u(x,t)$ is bounded by the line $|x|=Ct+b$ in the time interval $[0,T_c)$.
The comparison for all times can be proved with an iteration process in time.

\noindent$\bullet$ Regularity requirements. Using the smooth solutions of the approximate equations, the previous conclusions hold for any constructed weak solution.

\end{proof}

\noindent\textbf{Remark.} The following result about the free boundary is valid only for $s<1/2$ and for solutions with bounded and compactly supported initial data. The result is a direct consequence of the parabolic barrier study done in the previous section. Since that barrier does not depend explicitly on $m$ if $m\geq2$, the proof presented in \cite{CaffVaz} is valid here. By free boundary $\mathcal{FB}(u)$ we mean, the topological boundary of the support of the solution $S(u)=\overline{\{(x,t): u(x,t)>0\}}$.

\begin{cor}[\textbf{Growth estimates of the support}]
Let $u_0$ be bounded with $u_0(x)=0$ for $|x|>R$ for some $R>0$. If $(x,t) \in \mathcal{FB}(u)$ then $ x \leq R + C t^{1/(2-2s)}$, where  $C=C(\|u_0\|_{\infty},N, s)$.
\end{cor}

\subsection{Persistence of positivity}

This property is also interesting in the sense that avoids the possibility of degeneracy points for the solutions. In particular,  assuming that the solutions are continuous, it implies the non-shrinking of the support.  Due to the nonlocal character of the operator, the following theorem can be proved only for a certain class of solutions.

\begin{lemma}\label{lemmaRadialSym}
Let $u$ be a weak solution as constructed in Theorem \ref{Thm2PMFP} and assume that the initial data $u_0(x)$ is radially symmetric and non-increasing in $|x|$. Then $u(x,t)$ is also radially symmetric and non-increasing in $|x|$.
\end{lemma}
\begin{proof}
The operators in the approximate problem \eqref{model1Aprox} are invariant under rotation in the space variable. Since the solution of problem \eqref{model1Aprox} is unique, then we obtain that $u(x,t)$ is radially symmetric.
\end{proof}

\begin{theorem}\label{ThmPersisPosit}
Let $u$ be a weak solution as constructed in Theorem \ref{Thm2PMFP} and assume that it is a radial function of the space variable $u(|x|,t)$ and is non-increasing in $|x|$. If $u_0(x)$ is positive in a neighborhood of a point $x_0$, then $u(x_0,t)$ is positive for all times $t>0$.
\end{theorem}
\begin{proof}
A similar technique as the one presented in the tail analysis is used for this proof, but with what we call true subsolutions. Assume $u_0(x)\geq c>0$ in a ball $B_R(x_0)$. By translation and scaling we can also assume $c=R=1$ and $x_0=0$. Again, we will study a possible first contact point with a barrier that shrinks quickly in time, like
\begin{equation}
U(x,t)=e^{-at}F(|x|),
\end{equation}
with $F:\mathbb{R}_{\geq 0}\lra\mathbb{R}_{\geq 0} $ to be chosen later and $a>0$ large enough. Choose $F(0)=1/2$, $F(r )=0$ for $r\geq 1/2$ and $F'(r ) \leq 0$ for all $r\in \mathbb{R}_{\geq 0}$. The contact point $(x_c,t_c)$ is sought in $B_{1/2}(0)\times (0,\infty)$. By approximation we can assume that  $u$ is positive everywhere so there are no contact points at the parabolic border. At the possible contact point $(x_c,t_c)$ we have
$$u(x_c,t_c)=U(x_c,t_c), \quad u_t(x_c,t_c) \leq U_t(x_c,t_c)= -a U(x_c,t_c),$$
$$\nabla u(x_c,t_c)=\nabla U(x_c,t_c)=e^{-at_c}F'(|x_c|)\mathbf{e_r}, \quad \mathbf{e_r}=x_c/|x_c|.$$
We recall the equation
\[u_t=(m-1)u^{m-2} \nabla u \nabla p + u^{m-1}\Delta p.\]
Then at the contact point $(x_c,t_c)$ we have
\[-aU=U_t\geq u_t= (m-1)U^{m-2} \nabla U \overline{\nabla p} + U^{m-1}\overline{\Delta p}, \]
where $\overline{\Delta p}=\Delta p(x_c,t_c)$.  Then
\[
-a e^{-at_c}F(|x_c|) \geq (m-1) e^{-a(m-2)t_c} F(|x|)^{m-2} e^{-at_c}F'(|x_c|)\overline{p_r}+ e^{-a(m-1)t_c} F(|x_c|)^{m-1}\overline{\Delta p}.
 \]
According to \cite{CaffVaz} we know that the term $F'(|x|)\,\overline{p_r}\geq0$ and $\overline{\Delta p}$ is bounded uniformly. Therefore
\[
-a e^{-at_c}F(|x_c|) \geq e^{-a(m-1)t_c} F(|x_c|)^{m-1}\overline{\Delta p}.
\]
Simplifying and using that $m\ge 2$, $\overline{\Delta p}$ is bounded uniformly and also $F$ is bounded, we obtain
\begin{eqnarray*}
a  \leq -e^{-a(m-2)t_c} F(|x_c|)^{m-2}\overline{\Delta p} \le K e^{-a(m-2)t_c}  \le K.
\end{eqnarray*}
This is not true if $a>K$ and we arrive at a contradiction.
\end{proof}

\noindent \textbf{Remark.} There exist counterexamples on the persistence of positivity property when the  hypothesis of Theorem  \ref{ThmPersisPosit} are not satisfied. In \cite{CaffVaz} (Theorem 6.2) the authors construct an explicit counterexample by taking an initial data with not connected support.

\section{Infinite propagation speed in the case $1<m<2$ and $N=1$}\label{SectionInfinite}

In this section we will consider model \eqref{model1}
\begin{equation}\label{model1repeat}
\partial_t u= \partial_x \cdot (u^{m-1} \partial_x p), \quad p=(-\Delta)^{-s}u,
\end{equation}
for $x\in \mathbb{R}$, $t>0$ and $s\in (0,1)$. We take compactly supported initial data $u_0\ge 0$ such that $u_0 \in L^1_{\text{loc}}(\mathbb{R}).$ We want to prove infinite speed of propagation of the positivity set for this problem.
This is not easy, hence we introduce the integrated solution $v$, given by
\begin{equation}\label{def.v}
v(x,t)=\int_{-\infty}^x u(y,t)\, dy \ge 0 \quad \text{for }t>0, \ x \in \mathbb{R}.
\end{equation}
Therefore $v_x=u$ and $v(x,t)$ is a solution of the equation
\begin{equation}\label{IntegEq}
\partial_t v= -|v_x|^{m-1}\lapal v ,
\end{equation}
in some sense that we will make precise. The exponents $\alpha$ and $s$ are related by
$\alpha=1-s.$  The technique of the integrated solution has been extensively used in the standard Laplacian case to relate the porous medium equation with its integrated version, which is the $p$-Laplacian equation, always in 1D, with interesting results, see e.\,g. \cite{KaminVaz91}. The use of this tool in \cite{BilerKarchMonneau} for fractional Laplacians  in the case $m=2$ was novel and very fruitful.  We consider equation \eqref{IntegEq} with initial data
\begin{equation}\label{initialv0}
v(x,0)=v_0(x):=\int_{-\infty}^x u_0(x)\,dx  \quad \text{ for all }x \in \mathbb{R}.
\end{equation}
Note that $v(x,t)$ is a non-decreasing function in the space variable $x$. Moreover, since $u(x,t)$ enjoys the property of conservation of mass, then $v(x,t)$ satisfies (see Figure \ref{figinitialdata})
$$ \lim_{x\to -\infty}v(x,t)=0, \quad  \lim_{x\to +\infty}v(x,t)=M $$
for all $t\geq0$. We devote a separate study to the solution $v$ of the integrated problem \eqref{IntegEq} in Section \ref{SectionIntegrProbl}.
The validity of the maximum principle for equation \eqref{IntegEq} allows to prove a clean propagation theorem for $v$.

\begin{theorem}[\textbf{Infinite speed of propagation}]\label{ThmInfv}
Let $v$ be the solution of Problem \eqref{IntegEq}-\eqref{initialv0}, and assume that $u_0\ge 0$ is compactly supported. Then $0<v(x,t)<M$ for all $t>0$ and $x\in \R$.
\end{theorem}

The use of the integrated function is what forces us to work in one space dimension.
The result  continues the theory of the porous medium equation with potential pressure, by proving that model \eqref{model1repeat} has different propagation properties depending on the exponent $m$ by the ranges $m\ge 2$ and $1<m<2$. Such a behaviour is well known to be typical for the classical Porous Medium Equation $u_t=\Delta u^m$, recovered formally for $s=0$, which has finite propagation for $m>1$ and infinite propagation for $m\le 1$.  Therefore, our result is  unexpected, since it shows that for the fractional diffusion model the separation between finite and infinite propagation is moved to $m=2$.

\medskip

\noindent {\sc Proof of  Theorem \ref{ThmInfFiniteProp}, part b).} This weaker result follows immediately. In fact, in Theorem \ref{ThmInfv} we prove that $v(x,t)$  defined by \eqref{def.v} is positive for every $t>0$ if $x \in \R$.  Therefore for every $t>0$ there exist points $x$ arbitrary far from the origin such that $u(x,t)>0$.

If moreover, $u_0$ is radially symmetric and non-increasing in $|x|$ and $u$ inherits the symmetry and monotonicity properties of the initial data as proved in Lemma \ref{lemmaRadialSym}. This ensures that $u$ can not take zero values for any $x\in\R$ and $t>0$.

\qed

\subsection{Study of the integrated problem}

\noindent $\bullet$ \textbf{Connection between Model \eqref{model1repeat} and Model \eqref{IntegEq}}

We explain how the properties of the Model \eqref{model1repeat} with $N=1$ can be obtained via a study of the properties of the integrated equation \eqref{IntegEq}.
We consider equation \eqref{model1repeat} with compactly supported initial data $u_0$ such that $u_0 \geq 0$. Let us say that $\text{supp }u_0 \subset[-R,R]$, where $R>0.$
Therefore, the corresponding initial data to be considered for the integrated problem is $v_0(x) = \int_{-\infty}^{x} u_0(y)dy$, for all $x \in \mathbb{R}.$ Then $v_0:\mathbb{R}\to [0,\infty)$ and has the properties
\begin{equation}\label{v0Assump}
 v_0(x)=0 \text{ for }x<-R, \quad v_0(x)=M \text{ for }x>R, \quad v_0'(x)\ge 0\text{ for }x\in (-R, R),
\end{equation}
where $\R>0$ is fixed from the beginning and $M=\int_{\mathbb{R}}u_0(x)dx$ is the total mass.

\begin{figure}[h!]
	\begin{center}
		\includegraphics[width=\textwidth]{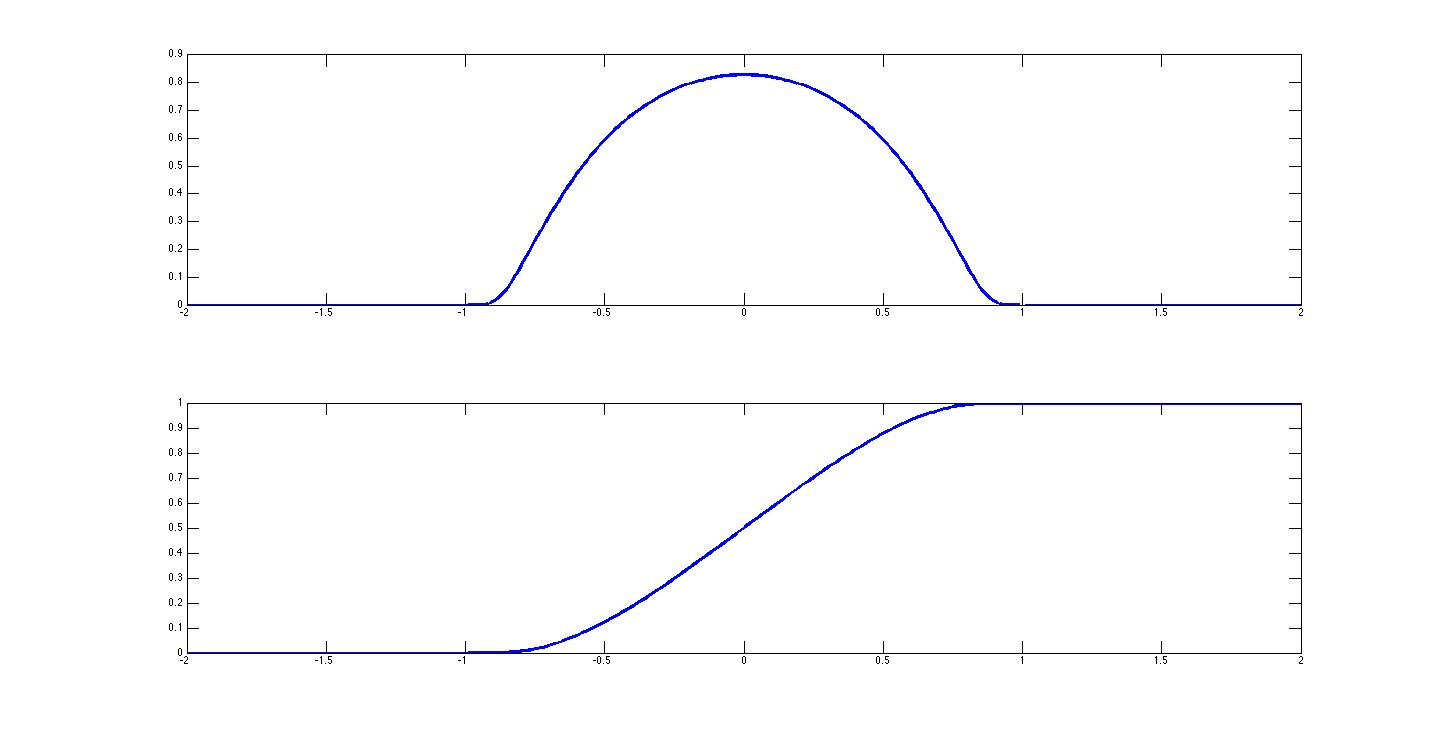}
		\caption{Typical compactly supported initial data for models (\ref{model1repeat}) and (\ref{IntegEq}).}
        		\label{figinitialdata}
	\end{center}
\end{figure}

\subsection{Regularity}

\begin{proposition}\label{propContxtforv}
The solution $v:[0,T]\times \R \to [0,\infty)$ of Problem \eqref{IntegEq} defined by formula\\
$\displaystyle{v(x,t)=\int_{-\infty}^x u(y,t)dy}$ is continuous in space and time.
\end{proposition}
\begin{proof}

\textbf{I. Preliminary estimates.} Since $v_x(x,t)=u(x,t)$, where $u$ is the solution of Problem \eqref{model1}, then by the estimates of Section \ref{SubsectEps} we have the following:

\noindent $\bullet$ $v_x =u \in L^\infty([0,T]: L^\infty(\R))$, therefore $v \in L^\infty([0,T]: \text{Lip}(\R))$, where $\text{Lip}(\R)$ is the space of Lipschitz continuous functions on $\R$. In particular, $v \in L^\infty([0,T]: \text{Lip}(B_R))$ for every $B_R\subset \R$.

\noindent $\bullet$ We have $(v_t)_x = u_t = \partial_x(u^{m-1}\partial_x (-\Delta)^{-s}u) $ in the sense of distributions.  Then $v_t \in L^2([0,T]:L^2(B))$ for every set $B\subset \R$, with $|B|<+\infty$. The proof is as follows. The first equality holds in the distributions sense, that is
$$\int_0^T\int_{\R} v_t \varphi_x  \, dx \, dt:=-\int_0^T\int_{\R} v (\varphi_x)_t  \, dx \, dt = \int_0^T\int_\R u^{m-1}\partial_x (-\Delta)^{-s}u \, \varphi_x \, dx \, dt, \quad \forall \varphi \in C^{\infty}_0(\R\times [0,T]).$$
This implies that $v_t=u^{m-1}\partial_x (-\Delta)^{-s}u $ a.e. in $\R$. Then, using the second energy estimate \eqref{SecondEnergy}, we obtain
\begin{align*}
\|v_t\|^2_{L^2([0,T]:L^2(B))} &= \|u^{m-1}\partial_x (-\Delta)^{-s}u\|^2_{L^2([0,T]:L^2(B))} \\
&\le \|u\|_{L^\infty(\R)}^{m-1} \int_0^T\int_{B}u^{m-1}|\partial_x(-\Delta)^{-s}u |^2 dx dt < +\infty.
\end{align*}

\noindent \textbf{II. Continuity in time.} Let $(x_0,t_0) \in \R \times [0,T].$ Let $(x,t_1) \in \R \times [0,T]$ and $h:=x-x_0$. Let $B=[x_0,x_1]$.
Then
$$|v(x_0,t_1)-v(x_0,t_0)|\le |v(x_0,t_1)-v(x,t_1)| +|v(x,t_0)-v(x_0,t_0)|+|v(x,t_1)-v(x,t_0)|.$$
We know $v\in \text{Lip}_x(\R)$; let $L$ the corresponding Lipschitz constant. Then
\begin{align*}|v(x_0,t_1)-v(x_0,t_0)| &\le 2L h + \frac{1}{h}\int_{x_0}^x |v(y,t_1)-v(y,t_0)|dy \\
&\le 2L h + \frac{1}{h}\int_{x_0}^x \left|\int_{t_0}^{t_1}v_t dt\right|dy \le 2L h^2 + \int_{x_0}^x \int_{t_0}^{t_1}\left|v_t\right|dydt \\
&\le 2L h + \frac{1}{h} |B|^{1/2}|t_1-t_0|^{1/2}\|v_t\|^2_{L^2([0,T]:L^2(B)} \\
&=2L h + \frac{|t_1-t_0|^{1/2}}{h^{1/2}} \|v_t\|^2_{L^2([0,T]:L^2(B)}.
\end{align*}
Optimizing, we choose $h\sim\frac{|t_1-t_0|^{1/2}}{h^{1/2}} $, that is $h \sim (t_1-t_0)^{3/2}$, and we obtain that
$$
|v(x_0,t_1)-v(x_0,t_0)| \le K |t_1-t_0|^{1/3}.$$
This estimate holds uniformly in $x\in \R$ and it proves that $v(x,t)$ is Hölder continuous in time.  In particular $v\in C([0,T]: C(\R))$.
\end{proof}

\subsection{Viscosity solutions}\label{SectionIntegrProbl}

\noindent\textbf{Notion of solution.} We define the notions of viscosity sub-solution, super-solution and solution in the sense of Crandall-Lions \cite{Crandall}. The definition will be adapted to our problem by considering the time dependency and also the nonlocal character of the Fractional Laplacian operator. For a presentation of the theory of viscosity solutions to more general integro-differential equations we refer to Barles and Imbert \cite{Barles}.

\noindent It will be useful to make the notations:

\noindent$\text{USC}(Q)=\{ \text{upper semi-continuous functions }u:Q \to \mathbb{R} \},$

\noindent$\text{LSC}(Q)=\{ \text{lower semi-continuous functions }u:Q \to \mathbb{R} \},$

\noindent$\text{C}(Q)=\{ \text{continuous functions }u:Q \to \mathbb{R} \}.$

\begin{defn}Let $v \in \text{USC} (\R \times (0,\infty))$ (resp. $v \in \text{LSC} (\mathbb{R} \times (0,\infty))$ ). We say that $v$ is a \textbf{viscosity sub-solution} (resp. \textbf{super-solution}) of equation \eqref{IntegEq} on $\R\times (0,\infty)$ if for any point $(x_0,t_0)$ with $t_0>0$ and any $\tau \in (0,t_0)$ and any test function $\varphi \in C^2( \mathbb{R} \times (0,\infty) ) \cap L^\infty( \mathbb{R}\times (0,\infty) )$ such that $v-\varphi$ attains a global maximum (minimum) at the point $(x_0,t_0)$ on
$$
Q_\tau=\mathbb{R} \times (t_0-\tau,t_0]
$$
we have that
$$
\partial_t \varphi (x_0,t_0)+ |\varphi_x(x_0,t_0)|^{m-1}(\lapal \varphi(\cdot, t_0)) (x_0) \le 0 \quad (\ge 0).
$$
\end{defn}
Since equation \eqref{IntegEq} is invariant under translation, the test function $\varphi$ in the above definition can be taken such that $\varphi$ touches $v$ from above in the sub-solution case, resp. $\varphi$ touches $v$ from below in the super-solution case.

We say that $v$ is a \textbf{viscosity sub-solution} (resp. \textbf{super-solution})
 of the initial-value problem \eqref{IntegEq}-\eqref{initialv0} on $\R\times (0,\infty)$ if it satisfies moreover at $t=0$
$$
v(x,0)\le  \limsup_{y\to x,\ t \to 0} v(y,t) \quad ( \text{resp. } v(x,0)\ge  \liminf_{y\to x,\ t \to 0} v(y,t)).
$$

We say that $v\in C(\R \times (0,\infty))$ is a \textbf{viscosity solution} if $v$ is a viscosity sub-solution and a viscosity super-solution on $\R\times (0,\infty)$.

\begin{proposition}[Existence of viscosity solutions]
Let $u$ be a weak solution for Problem \eqref{model1}. Then $v$ defined by formula $\displaystyle{v(x,t)=\int_{-\infty}^x u(y,t)dy}$ is a viscosity solution for Problem \eqref{IntegEq}-\eqref{initialv0}.
\end{proposition}
 \begin{proof}By Proposition \ref{propContxtforv} we know that $\displaystyle{v \in C([0,T]: C(\R))}$.
 The idea is to obtain a viscosity solution by the approximation process. 
 Let $v_\epsilon$ defined by  $\displaystyle{v_\epsilon(x,t)=\int_{-\infty}^x u_\epsilon(y,t)dy}$, where $ u_\epsilon$ is the approximation of $u$ as in Section \ref{SectionExistI}. Then $v_\epsilon$ is a classical solution, in particular a viscosity solution, to the problem
 $$(v_\epsilon)_t= \delta \Delta(v_\epsilon) + |(v_\epsilon)_x|^{m-1}(-\Delta)^{1-s}v_\epsilon.$$
 Since $u_\epsilon \to u$, then we get that $v_\epsilon \to v$ as $\epsilon \to 0$ (and similarly with respect to the other parameters). The final argument is to prove that a limit of viscosity solutions is a viscosity solution of Problem \eqref{IntegEq}-\eqref{initialv0}.

\end{proof}

The standard comparison principle for viscosity solutions holds true. We refer to Imbert, Monneau and Rouy \cite{ImbertHomog} where they treat the case $m=2$ and $\alpha=1/2$. Also, we mention Jakobsen and Karlsen \cite{Jakobsen} for the elliptic case.

\begin{proposition}[Comparison Principle]Let $m\in (1,2)$, $\alpha\in (0,1)$, $N=1$. Let $w$ be a sub-solution and $W$ be a super-solution in the viscosity sense of equation \eqref{IntegEq}.
If $w(x,0)\le v_0 \le W(x,0)$, then $w\le W$ in $\R \times (0,\infty)$.
\end{proposition}

We give now our extended version of parabolic comparison principle, which represents an important instrument when using barrier methods. This type of result is  motivated by the nonlocal character of the problem and the construction of lower barriers in a desired region $\Omega \subset \mathbb{R}$ possibly unbounded. This determines the parabolic boundary of a domain of the form $\Omega \times[0,T]$ to be $(\mathbb{R}\setminus \Omega) \times [0,T] \cup \mathbb{R} \times \{0\}$, where $\Omega \subset \mathbb{R}$. A similar parabolic comparison has been proved in \cite{CabreRoquejoffre} and has been used for instance in \cite{CabreRoquejoffre,StanVazquezKPP}.

\begin{proposition}\label{ComparPrinc}
Let $m>1$, $\alpha \in (0,1)$.  Let $v$ be a viscosity solution of Problem \eqref{IntegEq}-\eqref{initialv0}. Let  $\Phi: \mathbb{R}\times[0,\infty)\to \mathbb{R}$ such that $\Phi \in C^2(\Omega\times (0,T))$. Assume that
\begin{itemize}
  \item $\Phi_t +|\Phi_x|^{m-1}\lapal\Phi< 0 $ for $x\in \Omega$, $t\in [0,T]$;
  \item $\Phi(x,0) < v(x,0)$  for all $x\in \mathbb{R}$ (comparison at initial time);
  \item $\Phi(x,t) < v(x,t)$  for all $x \in \mathbb{R} \setminus \Omega$ and $t\in (0,T)$ (comparison on the parabolic boundary).
\end{itemize}
Then $\Phi(x,t) \le v(x,t)$ for all $x \in \mathbb{R}$, $t\in (0,T).$
\end{proposition}
\begin{proof}
The proof relies on the study of the difference $\Phi-v:\mathbb{R}\times[0,\infty)\to \mathbb{R}$. At the initial time $t=0$ we have by hypothesis that $\Phi(x,0) - v(x,0) < 0$ for all $x\in \mathbb{R}$.

Now, we argue by contradiction. We assume that the function $\Phi-v$ has a first contact point $(x_c,t_c)$ where $x_c\in \Omega$ and $t_c \in (0,T)$. That is, $(\Phi-v)(x_c,t_c)=0$ and $(\Phi-v)(x,t)<0$ for all $0<t<t_c$, $x \in \mathbb{R}$, by regularity assumptions. Therefore, $(\Phi-v)$ has a global maximum point at $(x_c,t_c)$ on $\R \times (0,t_c].$ Therefore, $v-\Phi$ attains a global minimum at $(x_c,t_c).$

Since $v$ is a viscosity solution and $\Phi$ is an admissible test function then by definition
$$
\Phi_t (x_c,t_c) +|\Phi_x (x_c,t_c) |^{m-1}\lapal\Phi (x_c,t_c) \ge 0 ,
$$
which is a contradiction since this value is negative by hypothesis.
\end{proof}

\subsection{Self-Similar Solutions. Formal approach }\label{sectionSelfSim}

Self-similar solutions are the key tool in describing the asymptotic behaviour of the solution to certain parabolic problems. We perform here a formal computation of a type of self-similar solution to equation \eqref{IntegEq}, being motivated by the construction of suitable lower barriers.

\noindent Let $m\in (1,2)$ and $\alpha\in (0,1).$  We search for self-similar solutions to equation \eqref{IntegEq} of the form $$U(x,t)=\Phi(|y|t^{-b})$$ which solve equation \eqref{IntegEq} in $\R \times (0,\infty).$ After a formal computation, it follows that the exponent $b>0$ is given by $b=1/(m-1+2\alpha)$ and the profile function $\Phi$ is a solution of the equation
\begin{equation*}\label{ProfileEq2}
by\Phi'(y)- |\Phi'(y)|^{m-1}\lapal \Phi(y)=0.
\end{equation*}
We deduce that any possible behaviour of the form $\Phi(y)=c|y|^{-\gamma}$ with $\gamma>1$ is given by
\begin{equation}\label{gamma}
\gamma=\frac{2\alpha+m}{2-m}.
\end{equation}
The value of the self-similarity exponent will be used in the next section for the construction of a lower barrier.  A further analysis of self-similar solutions is beyond the purpose of this paper and can be the subject of a new work. We mention that in the case $m=2$, the profile function $\Phi$ has been computed explicitly by Biler, Karch and Monneau in \cite{BilerKarchMonneau}.

\subsection{Construction of the lower barrier}
 In this section we present a class of sub-solutions of equation \eqref{IntegEq} which represent an important tool in the proof of the infinite speed of propagation. For a suitable choice of parameters this type of sub-solution will give us a lower bound for $v$ in the corresponding domain. This motivates us to refer to this function as a lower barrier. We mention that a similar lower barrier has been constructed in \cite{StanVazquezKPP}.

Let $\displaystyle{\gamma=\frac{m+2\alpha}{2-m}}\text{ and }b=\frac{1}{m-1+2\alpha}$ be the exponents deduced in Section \ref{sectionSelfSim}.

We fix $x_0<0$. In the sequel we will use as an important tool a function $G:\mathbb{R}\to \mathbb{R}$ such that, given any two constants $C_1>0$ and $C_2>0$, we have that
\begin{itemize}
  \item (G1) $G$ is compactly supported in the interval $(-x_0,\infty)$;
  \item (G2) $G(x) \le C_1$ for all $x\in \mathbb{R}$;
  \item (G3) $(-\Delta)^s G(x) \le -C_2|x|^{-(1+2s)}$ for all $x<x_0$.
\end{itemize}
This technical result will be proven in Lemma \ref{LemmaG} of Section \ref{sectionApend} (Appendix).

\begin{lemma}[\textbf{Lower Barrier}]\label{LemmaBarrier} Let $x_0<0$, $\epsilon>0$ and $ \xi>0$. Also, let $G$ be a function with the properties (G1),(G2) and (G3). We consider the barrier
\begin{equation}\label{barrier1}
\Phi_\epsilon(x,t)=(t+\tau)^{b \gamma}\left((|x|+\xi)^{-\gamma}+G(x)\right)-\epsilon, \quad t\geq0, \ x\in \mathbb{R}.
\end{equation}
Then for a suitable choice of the parameter $C_2>0$, the function $\Phi_\epsilon$ satisfies
\begin{equation}\label{ineqPhi}
(\Phi_\epsilon)_t  + |(\Phi_\epsilon)_x|^{m-1} \lapal \Phi_\epsilon \le 0 \quad \text{ for }x<x_0, \ t> 0.
\end{equation}
Moreover, $C_1$ is a free parameter and $C_2=C_2(N,m,\alpha,\tau)$.
\end{lemma}
\begin{proof}
We start by checking under which conditions $\Phi_\epsilon$ satisfies \eqref{ineqPhi}, that is, $\Phi$ is a classical sub-solution of equation \eqref{IntegEq} in $Q$. To this aim, we have that
\begin{align*}
&(\Phi_\epsilon)_t  + |(\Phi_\epsilon)_x|^{m-1} \lapal \Phi_\epsilon =b \gamma \frac{(t+\tau)^{b \gamma-1}}{(|x|+\xi)^{\gamma}} +
\gamma^{m-1}\frac{(t+\tau)^{b \gamma(m-1)}}{(|x|+\xi)^{(\gamma+1)(m-1)}} \lapal \Phi_\epsilon(x,t) \\
&=b \gamma \frac{(t+\tau)^{b \gamma-1}}{(|x|+\xi)^{\gamma}} +
\gamma^{m-1}\frac{(t+\tau)^{b \gamma(m-1)}}{(|x|+\xi)^{(\gamma+1)(m-1)}} (t+\tau)^{b\gamma}
\bigg(\lapal [ (|x| +\xi)^{-\gamma}]+ \lapal G\bigg) .
\end{align*}
Now, by Lemma \ref{LemmaFelix} we get the estimate $\lapal \left((|x|+\xi)^{-\gamma} \right)\le  C_3 |x|^{-(1+2\alpha)}$ for all $|x|\ge |x_0|$, with positive constant $C_3=C_3(N,m,\alpha).$
At this step, we choose the parameter $C_2$ in the assumption (G2) to be at least $C_2>C_3$. The precise choice will be deduced later.  Since $\gamma=(\gamma+1)(m-1)+1+2\alpha$, we continue as follows:
\begin{align*}
(\Phi_\epsilon)_t  &+ |(\Phi_\epsilon)_x|^{m-1} \lapal \Phi_\epsilon\\
&\le b \gamma \frac{(t+\tau)^{b \gamma-1}}{(|x|+\xi)^{\gamma}} +
\gamma^{m-1}\frac{(t+\tau)^{b \gamma m}}{(|x|+\xi)^{(\gamma+1)(m-1)}}  ( C_3-C_2)|x|^{-(1+2\alpha)} \\
&={(|x|+\xi)^{-(\gamma+1)(m-1)}} \cdot \\
&\quad \cdot
\left(  b \gamma (t+\tau)^{b \gamma-1} {(|x|+\xi)^{-(1+2\alpha)}} + \gamma^{m-1} (t+\tau)^{b \gamma m}(C_3-C_2)|x|^{-(1+2\alpha)} \right)\\
&\le (|x|+\xi)^{-(\gamma+1)(m-1)}|x|^{-(1+2\alpha)}\left(  b \gamma (t+\tau)^{b \gamma-1}  + \gamma^{m-1} (t+\tau)^{b \gamma m}(C_3-C_2) \right)
\end{align*}
which is negative for all $(x,t)\in Q$, if we ensure that $C_2$ is such that:
\begin{equation}\label{C2}
C_2>C_3+b\gamma^{2-m}\tau^{b\gamma(1-m)-1}.
\end{equation}
This choice of $C_2$ is independent on the parameters $\xi,\ \epsilon.$
\end{proof}

From now on, we will take $\tau=1$, which will be enough for our purpose.
We can now prove the main result for the model \eqref{IntegEq} which in particular implies the infinite speed of propagation of model \eqref{model1} for $1<m<2$ in dimension $N=1$.

\subsection{Proof Theorem \ref{ThmInfv}}
 Let $x_0<0$ fixed. We prove that $v(x,t)>0$ for all $t>0$ and $x<x_0$. By scaling arguments, the initial data $v_0$ with properties \eqref{v0Assump}, satisfies
\begin{equation}\label{u0Heaviside}v_0(x)\ge H_{x_0}(x)=\left\{
  \begin{array}{ll}
    0, & \hbox{$x< x_0$,} \\[2mm]
  1, & \hbox{$x>x_0$.}
  \end{array}\right.\end{equation}

We will prove that $v(x,t) \ge \Phi_\epsilon(x,t)$ in the parabolic domain $Q_T=\{x<x_0, \ t \in [0,T]\}$ by using as an essential tool the Parabolic Comparison Principle established in Proposition \ref{ComparPrinc}.  We describe the proof in the graphics below, where the barrier function is represented, for simplicity, without the modification caused by the function $G(\cdot)$ (Figure \ref{FigBarrier1}).

To this aim we check the required conditions in order to apply the above mentioned comparison result.

\noindent $\bullet$ {\bf  Comparison on the parabolic boundary.} This will be done in two steps.

\noindent (a) \textbf{Comparison at the initial time.} The initial data \eqref{u0Heaviside} naturally impose the following conditions on $\Phi_\epsilon$. At time $t=0$ we have $\Phi_\epsilon(x_0,0)<0$, which holds only if $\xi$ satisfies
\begin{equation}\label{cond1xi}
 \xi > x_0+ \epsilon^{-\frac{1}{\gamma}}.
\end{equation}
Therefore $\Phi_\epsilon(x_0,0) <v_0(x_0)$ since $v_0(x_0)>0$.

\noindent (b) \textbf{Comparison on the lateral boundary.} Let $\displaystyle k_1 :=\min\{ v(x,t): \ x\ge x_0, \ 0< t\le T \}$ with $k_1>0.$  This results follows from the continuity $v\in C([0,T]:C(\R))$ since $v_0(x_0)=1$.
We impose the condition
$$
\Phi_\epsilon(x,t) < v(x,t) \quad \text{ for all }x\ge x_0, \ t\in [0,T].
$$
It is sufficient to have
$$
 (T+1)^{b \gamma} (\xi^{-\gamma}+C_1 ) < k_1.
$$
The maximum value of $T$ for which this inequality holds is
\begin{equation}\label{Tmax}
T<  \left(\frac{k_1}{ \xi^{-\gamma}+C_1 } \right)^{1/b\gamma}-1.
\end{equation}
We need to impose a compatibility condition on the parameters in order to have $T>0$, that is:
\begin{equation}\label{cond2xi}
\xi > (k_1-C_1)^{-\frac{1}{\gamma}}.
\end{equation}
The remaining parameter $C_1$ from assumption (G2) is chosen here such that: $C_1<k_1$.

By Proposition \ref{ComparPrinc} we obtain the desired comparison
$$
v(x,t) \ge \Phi_\epsilon(x,t) \quad \text{for all } (x,t) \in Q_T.
$$

\noindent $\bullet$ {\bf Infinite speed of propagation}. Let $x_1 < x_0$ and $t_1 \in (0,T)$ where $T$ is given by \eqref{Tmax}. We prove there exists a suitable choice of $\xi$ and $\epsilon$ such that $\Phi_\epsilon(x_1,t_1)>0$. This is equivalent to impose the following upper bound on $\xi$:
\begin{equation}\label{cond3xi}
 \xi < x_1 + (t_1+1)^b \epsilon^{-\frac{1}{\gamma}}.
\end{equation}
We need to check now if there exists $\epsilon>0 $ such that condition \eqref{cond3xi} is compatible with conditions \eqref{cond1xi} and \eqref{cond2xi}. For the compatibility of conditions  \eqref{cond1xi} and \eqref{cond2xi} we have
\[
x_0+ \epsilon^{-\frac{1}{\gamma}} <\xi< x_1 + (t_1+1)^b \epsilon^{-\frac{1}{\gamma}},
\]
that is,
\begin{equation}\label{epsilon1}
\epsilon < \left[\frac{(t_1+1)^b-1}{x_0-x_1}\right]^\gamma.
\end{equation}
For conditions \eqref{cond2xi} and \eqref{cond3xi} we need
\[
(k_1-C_1)^{-\frac{1}{\gamma}} \leq \xi < x_1 + (t_1+1)^b \epsilon^{-\frac{1}{\gamma}},
\]
which is equivalent to
\begin{equation}\label{epsilon2}
\epsilon  < \left[\frac{(t_1+1)^b}{(k_1-C_1)^{-\frac{1}{\gamma}} -x_1}\right]^\gamma.
\end{equation}
Both upper bounds \eqref{epsilon1} and \eqref{epsilon2} make sense since $0>x_0>x_1$ and $k_1>C_1$.

\noindent\textbf{Summary.} The proof was performed in a constructive manner and we summarize it as follows:
$C_1<k_1$, $T$ given by \eqref{Tmax}. Then by taking  $\epsilon$ the minimum of \eqref{epsilon1}-\eqref{epsilon2}, $\xi$ satisfying \eqref{cond1xi}-\eqref{cond2xi}-\eqref{cond3xi} we obtain that $\Phi(t_1,x_1)>0$.

This proofs that $v(t_1,x_1)>0$ for any $t\in (0,T).$

\qed

\definecolor{ffqqqq}{rgb}{1,0,0}
\definecolor{qqttcc}{rgb}{0,0.2,0.8}
\definecolor{sqsqsq}{rgb}{0.125,0.125,0.125}

\definecolor{ffqqqq}{rgb}{1,0,0}
\definecolor{qqqqcc}{rgb}{0,0,0.8}
\definecolor{sqsqsq}{rgb}{0.125,0.125,0.125}

\begin{figure}[ht!]
\centering
\begin{tikzpicture}[line cap=round,line join=round,>=triangle 45,x=7.5cm,y=4.2cm]
\clip(-0.896,-0.128) rectangle (1.145,1.11);
\draw [line width=1.6pt,color=sqsqsq] (0,-0.128) -- (0,1.11);
\draw [line width=1.2pt,color=sqsqsq,domain=-0.896:1.145] plot(\x,{(-0-0*\x)/5.977});
\draw [dash pattern=on 1pt off 1pt,color=sqsqsq,domain=-0.896:1.145] plot(\x,{(--2-0*\x)/2});
\draw [line width=2.8pt,color=qqttcc,domain=-0.896473504307193:-0.27353490278664827] plot(\x,{(-0-0*\x)/-0.471});
\draw [dash pattern=on 1pt off 1pt,color=sqsqsq] (-0.274,0)-- (-0.274,1);
\draw [line width=2.8pt,color=qqttcc,domain=-0.27353490278664816:1.145157305834116] plot(\x,{(--0.844-0*\x)/0.844});
\draw [dash pattern=on 1pt off 1pt,color=sqsqsq,domain=-0.896:1.145] plot(\x,{(--0.092-0*\x)/-0.92});
\draw [shift={(-0.575,0.493)},line width=2pt,color=ffqqqq]  plot[domain=4.749:6.18,variable=\t]({1*0.579*cos(\t r)+0*0.579*sin(\t r)},{0*0.579*cos(\t r)+1*0.579*sin(\t r)});
\draw [line width=2pt,color=ffqqqq] (-0.554,-0.085)-- (-1.424,-0.09);
\draw [shift={(0.584,0.496)},line width=2pt,color=ffqqqq]  plot[domain=3.248:4.654,variable=\t]({1*0.587*cos(\t r)+0*0.587*sin(\t r)},{0*0.587*cos(\t r)+1*0.587*sin(\t r)});
\draw [line width=2pt,color=ffqqqq,domain=0.55:1.145157305834116] plot(\x,{(-0.087-0.005*\x)/0.989});
\draw (1.047,0.003) node[anchor=north west] {$\textcolor{black}{x}$};
\draw (0.1,0.24) node[anchor=north west] {$\textcolor{black}\Phi_\epsilon(x,0)$};
\draw (-0.561,0.15) node[anchor=north west] {$\textcolor{black}v_0(x)$};
\draw (-0.051,1.12) node[anchor=north west] {\textcolor{black}1};
\draw (-0.051,0.096) node[anchor=north west] {\textcolor{black}0};
\draw (-0.09,-0.025) node[anchor=north west] {$\textcolor{black}-\mathbb{\epsilon}$};
\begin{scriptsize}
\fill [color=sqsqsq] (-0.274,0) circle (1.5pt);
\draw[color=sqsqsq] (-0.245,-0.03) node {$x_0$};
\end{scriptsize}
\end{tikzpicture}
\caption{Comparison with the barrier at time $t=0$}
\label{FigBarrier0}
\end{figure}
\begin{figure}
\centering
\begin{tikzpicture}[line cap=round,line join=round,>=triangle 45,x=7.5cm,y=4.2cm]
\clip(-0.896,-0.128) rectangle (1.145,1.11);
\draw [line width=1.6pt,color=sqsqsq] (0,-0.128) -- (0,1.11);
\draw [line width=1.2pt,color=sqsqsq,domain=-0.896:1.145] plot(\x,{(-0-0*\x)/5.977});
\draw [dash pattern=on 1pt off 1pt,color=sqsqsq,domain=-0.896:1.145] plot(\x,{(--2-0*\x)/2});
\draw [dash pattern=on 1pt off 1pt,color=sqsqsq,domain=-0.896:1.145] plot(\x,{(--0.092-0*\x)/-0.92});
\draw (1.047,0.003) node[anchor=north west] {$\textcolor{black}{x}$};
\draw (0.35,0.21) node[anchor=north west] {$\textcolor{black}\Phi_\epsilon(x,t_1)$};
\draw (-0.051,1.12) node[anchor=north west] {\textcolor{black}1};
\draw (-0.051,0.096) node[anchor=north west] {\textcolor{black}0};
\draw (-0.09,-0.025) node[anchor=north west] {$\textcolor{black}-\mathbb{\epsilon}$};
\draw [shift={(-0.74,0.598)},line width=2pt,color=qqqqcc]  plot[domain=4.721:5.917,variable=\t]({1*0.564*cos(\t r)+0*0.564*sin(\t r)},{0*0.564*cos(\t r)+1*0.564*sin(\t r)});
\draw [shift={(0.604,0.103)},line width=2pt,color=qqqqcc]  plot[domain=1.633:2.798,variable=\t]({1*0.868*cos(\t r)+0*0.868*sin(\t r)},{0*0.868*cos(\t r)+1*0.868*sin(\t r)});
\draw [line width=2pt,color=qqqqcc,domain=-0.896473504307193:-0.7353358406914835] plot(\x,{(-0.005-0*\x)/-0.138});
\draw [line width=2pt,color=qqqqcc,domain=0.55:1.145157305834116] plot(\x,{(--0.444--0.009*\x)/0.462});
\draw [shift={(-0.879,0.826)},line width=2pt,color=ffqqqq]  plot[domain=4.667:6.081,variable=\t]({1*0.897*cos(\t r)+0*0.897*sin(\t r)},{0*0.897*cos(\t r)+1*0.897*sin(\t r)});
\draw [shift={(0.879,0.826)},line width=2pt,color=ffqqqq]  plot[domain=3.343:4.759,variable=\t]({1*0.897*cos(\t r)+0*0.897*sin(\t r)},{0*0.897*cos(\t r)+1*0.897*sin(\t r)});
\draw [line width=2pt,color=ffqqqq,domain=0.92:1.145157305834116] plot(\x,{(-0.014-0.001*\x)/0.215});
\draw (-0.55,0.3) node[anchor=north west] {$\textcolor{black}v(x,t_1)$};
\begin{scriptsize}
\fill [color=sqsqsq] (-0.274,0) circle (1.5pt);
\draw[color=sqsqsq] (-0.245,-0.03) node {$x_0$};
\fill [color=sqsqsq] (-0.529,0) circle (1.5pt);
\draw[color=sqsqsq] (-0.499,-0.03) node {$x_1$};
\end{scriptsize}
\end{tikzpicture}\caption{Comparison with the barrier at time $t>0$} \label{FigBarrier1}

\end{figure}


\medskip

\noindent\textbf{Remark.} The parameter $\xi$ of the barrier depends on $\epsilon$ by \eqref{cond1xi} and \eqref{cond3xi} and therefore $\xi \to \infty$ as $\epsilon \to 0.$
Therefore $\Phi_\epsilon(x,t) \to 0$ as $\epsilon \to 0$ for every $(x,t) \in Q_T$ and we can not derive a lower parabolic estimate for $v(x,t)$ in $Q_T.$


\section{Appendix}\label{sectionApend}

\subsection{Estimating the Fractional Laplacian}
In this section we are interested in estimating the fractional Laplacian of given functions. We recall the definition of the Fractional Laplacian operator
\begin{equation*}
(-\Delta)^{s}u(x)=\sigma_{N,s} \ P.V. \int_{\RN}\frac{u(x)-u(y)}{|x-y|^{N+2s}}, \quad 0<s<1,
\end{equation*}
where $\sigma_s$ a normalization constant given by
\begin{equation*}
\sigma_{N,s}=\frac{2^{2s}\Gamma(\frac{N+2s}{2})}{\pi^{N/2}\Gamma(-N/2)}.
\end{equation*}

First, given the expression of the fractional Laplacian, we construct a function with the desired properties.

\begin{lemma}\label{LemmaG}
Given two arbitrary constants $C_1, C_2>0$ there exists a function $G: \mathbb{R}\to [0,+\infty)$ with the following properties:
\begin{enumerate}
  \item $G$ is compactly supported.
  \item $G(x) \le C_1$ for all $x\in \mathbb{R}$
  \item $(-\Delta)^s G(x) \le -C_2|x|^{-(1+2s)}$ for all $x \in \mathbb{R}$ with $d(x, \mbox{supp}(G))\ge 1.$
\end{enumerate}
\end{lemma}
\begin{proof}
Let $R$ an arbitrary positive number to be chosen later.
We consider a smooth function  $G_1: \mathbb{R}\to [0,+\infty)$ such that $G_1(x) \le C_1$ for all $x \in \mathbb{R}$ and supported in the interval $[-1,1]$.

We define $G_R(x)=G_1(x/R)$. Therefore $\|G_R\|_{L^1(\mathbb{R})} = R \|G_1\|_{L^1(\mathbb{R})}$, $G_R \le C_1$ and $G$ is supported in the interval $[-R,R]$. Then for $|x|\ge R+1$ we have that
\begin{align*}
(-\Delta)^s G_R(x)&=\sigma_{s} \int_{\mathbb{R}}\frac{G_R(x)-G_R(y)}{|x-y|^{1+2s}} dy = -\sigma_{s} \int_{-R}^R\frac{G_R(y)}{|x-y|^{1+2s}} dy \\
&\le -\sigma_{s} \int_{-R}^R \frac{G_R(y)}{(|x|+R)^{1+2s}} dy =-\sigma_s (|x|+R)^{-(1+2s)}\|G_R\|_{L^1(\mathbb{R})} \\
&\le -\sigma_{s} 2^{-(1+2s)}\|G_R\|_{L^1(\mathbb{R})}|x|^{-(1+2s)}=-\sigma_{s} 2^{-(1+2s)}R\|G_1\|_{L^1(\mathbb{R})}|x|^{-(1+2s)}.
\end{align*}
It is enough to choose $\displaystyle R\geq \frac{C_2 2^{1+2s}}{\sigma_s ||G_1||_{L^1(\mathbb{R})}}$ to get $(-\Delta)^s G_R(x)\leq C_2 |x|^{-(1+2s)}$.  Note that $R$ implicitly depends on $C_1$ since $||G_1||_{L^1(\mathbb{R})}\leq 2C_1$.

\end{proof}

Secondly, we need to estimate the fractional Laplacian of a negative power function. The following result is similar to one proven by Bonforte and V\'{a}zquez in Lemma 2.1 from \cite{BV2012} with the main difference that our function is $C^2$ away from the origin. We make a brief adaptation of their proof to our situation.

\begin{lemma}\label{LemmaFelix}
Let $\varphi :\mathbb{R}\to (0,\infty)$, $\varphi=(|x|+\xi)^{-\gamma}$, where $\gamma>1$ and $\xi>0.$
  Then, for all $|x|\ge |x_0|>1 $, we have that
\begin{equation}
|(-\Delta)^s \varphi(x)| \le \frac{C}{|x|^{1+2s}},
\end{equation}
with positive constant $C>0$ that depends only on $\gamma,\ \xi, \ s$.
\end{lemma}
\begin{proof}

Let us first estimate the $L^1$ norm of $\varphi$.
\begin{align*}
\int_{\R}\varphi(x) dx &= \int_{|x|<1}\varphi(x) dx +\int_{|x|>1}\varphi(x) dx \le \int_{|x|<1}\xi^{-\gamma} dx + \int_{|x|>1}x^{-\gamma} dx \\
&\le 2 \xi^{-\gamma} + 2\int_{1}^{\infty}r^{-\gamma}dr = 2 \xi^{-\gamma} + \frac{2}{\gamma-1} <C, \quad C= C(\gamma, \xi).
\end{align*}
Following the ideas of \cite{BV2012} Lemma 2.1, the computation of the $(-\Delta)^s \varphi(x)$ is based on estimating the integrals on the regions
$$
R_1=\{y: \ |y| > 3|x|/2 \},\quad R_2=\left\{ y: \ \frac{|x|}{2} < |y| < \frac{3|x|}{2} \right\} \setminus B_{\frac{|x|}{2}}(x) ,
$$
$$
R_3=\{ y: \ |x-y| < |x|/2 \} ,\quad R_4=\{ y: \ |y| < |x|/2 \} .
$$
Therefore
$$
(-\Delta)^s \varphi(x)=\int_{R_1\cup R_2\cup R_3 \cup R_4} \frac{\varphi(x)-\varphi(y)}{|x-y|^{1+2s}} dy.
$$
We proceed with the estimate of each of the four integrals:
\begin{align*}
I&=\int_{|y| > 3|x|/2 } \frac{\varphi(x)-\varphi(y)}{|x-y|^{1+2s}} dy \le \omega_d \varphi(x) \int_{ 3|x|/2 }^{\infty}\frac{dr}{r^{1+2s}} = \frac{K_1}{|x|^{\gamma+2s}}, \quad K_1=K_1(\gamma,s).
\end{align*}
\begin{align*}
II&=\int_{R_2} \frac{\varphi(x)-\varphi(y)}{|x-y|^{1+2s}} dy \le \frac{\varphi(x)}{(|x|/2)^{1+2s}} \int_{ |x|/2 }^{ 3|x|/2 }dr =\frac{K_2}{|x|^{\gamma+2s}},\quad K_2=K_2(\gamma,s).
\end{align*}
\begin{align*}
III&=\int_{R_3} \frac{\varphi(x)-\varphi(y)}{|x-y|^{1+2s}} dy \le \|\varphi''\|_{L^\infty(B_{|x|/2}(x) )} \int_{|x-y|\le|x|/2} \frac{1}{|x-y|^{2s-1}}dy \\
&\le \frac{K_3'}{|x|^{\gamma+2}} \int_0^{|x|/2} \frac{1}{r^{2s-1}}dr \le \frac{K_3}{|x|^{\gamma+2s}},\qquad \qquad \qquad \qquad\qquad K_3=K_3(\gamma,s).
\end{align*}
For $IV$ we take use that when $|y|<|x|/2$ then $|y-x|\ge  |x|/2$ and $|y|<|x|$ which implies $\varphi(y)>\varphi(x).$ We have
\begin{align*}
IV&\le \int_{|y|<|x|/2} \frac{|\varphi(x)-\varphi(y)|}{|x-y|^{1+2s}}dy\le \left(\frac{2}{|x|}\right)^{1+2s}\int_{|y|<|x|/2} \varphi(y)dy \le \left(\frac{2}{|x|}\right)^{1+2s}\|\varphi\|_{L^1(\R)} \\
&\le \frac{K_4}{|x|^{1+2s}}, \quad K_4= K_4(\gamma,s,\xi).
\end{align*}
Since $\gamma>1$, we  can conclude that
$$|(-\Delta)^s \varphi(x)| \le |I|+|II|+|III|+|IV|= K_5 |x|^{-\gamma-2s} + K_4 |x|^{-1-2s} \le K_6 |x|^{-1-2s}, \quad \forall |x|\ge|x_0| >1.$$
\end{proof}

\subsection{Reminder on cut-off functions}\label{cutoffSect}

We remind the construction of cut-off functions. Let $$f(x)= \left\{
\begin{array}{ll}
 e^{-1/x},& \, x \ge 0,\\[2mm]
0, & \, x<0.
\end{array}
\right. $$
Then $f\in C^\infty(\R)$. Let
$$F(x)=\frac{f(x)}{f(x)+f(1-x)}, \quad x \in \R.$$
Then $F(x)=0$ for $x<0$, $F(x)=1$ for $x\ge 1$ and $F(x)\in (0,1)$ for $x\in (0, 1).$
We construct now the cut-off function $\varphi:\RN \to ([0,1]$ by:
$$\varphi(x) =  F(2-|x|), \, \x \in \RN.$$
Then $\varphi \in C^\infty(\RN)$, $\varphi(x) =1$ for $|x|\le 1$, $\varphi(x)=0$ for $|x|\ge 2$ and $\varphi(x) \in (0,1)$ for $|x|\in (1,2)$.
The cut-off function for $B_R$ is obtained by
$$\varphi_R(x)=\varphi(x/R).$$ Thus $\varphi_R \in C^\infty(\RN)$, $\varphi_R(x) =1$ for $|x|\le R$, $\varphi_R(x)=0$ for $|x|\ge 2R$ and $\varphi_R(x) \in (0,1)$ for $|x|\in (R,2R)$.
Also, we have that $\nabla (\varphi_R) =O(R^{-1})$, $\Delta(\varphi_R) =O(R^{-2}).$

\subsection{Compact sets in the space $L^p(0,T;B)$}

Necessary and sufficient conditions of convergence in the spaces $L^p(0,T;B)$ are given by Simon in \cite{simon}. We recall now their applications to evolution problems. We consider the spaces $X\subset B\subset Y$ with compact embedding $X\rightarrow B$.

\begin{lemma}\label{ConvSimon1}Let $\mathcal{F}$ be a bounded family of functions in $L^p(0,T;X)$, where $1 \leq p <\infty$ and $\partial \mathcal{F}/\partial t=\{\partial f/\partial t: f \in \mathcal{F}\}$ be bounded in  $L^1(0,T;Y)$. Then the family $\mathcal{F}$ is relatively compact in $L^p(0,T;B)$.
\end{lemma}

\begin{lemma}\label{ConvSimon2}Let $\mathcal{F}$ be a bounded family of functions in $L^\infty(0,T;X)$ and $\partial \mathcal{F}/\partial t$ be bounded in  $L^r(0,T;Y)$, where $r>1$. Then the family $\mathcal{F}$ is relatively compact in $C(0,T;B)$.
\end{lemma}

\section{Comments and open problems}\label{sec.comm}
\noindent $\bullet$ \textbf{Case $m\ge 3$.} In this range of  exponents the first energy estimate does not hold anymore. Therefore, we lose the compactness result needed to pass to the limit in the approximations to obtain a weak solution of the original problem.  The second energy estimate is still true and it gives us partial results for compactness. In our opinion a suitable tool to replace the first energy estimate would be proving the decay of some $L^p$ norm. In  that case we will also need a Stroock-Varoupolous type inequality for some approximation $\mathcal{L}_s^\epsilon $ of the fractional Laplacian.  The technique of regularizing the kernel by convolution that we have used through this paper does not allow us to prove such kind of inequality. The idea is however to use a different approximation of the pressure term that is well suited to the  Stroock-Varoupolous type inequalities. Let us mention \cite{EndalJacobsenTeso} where this kind of inequalities are proved for a wider class of nonlocal operators including $\mathcal{L}_s^\epsilon $. The technical details  are involved
and the new approximation may have an interest, hence we think it deserves a separate study.

\medskip

\noindent $\bullet$\textbf{ Infinite propagation in higher dimensions for self similar solutions.} In \cite{StanTesoVazTrans} we proved a transformation formula between self-similar solutions of the model \eqref{model1} with $1<m<2$ and the fractional porous medium equation $u_t+(-\Delta)^s u^m=0$. This way we obtain infinite propagation for self similar solutions of the form $U(x,t)=t^{-\alpha}F(|x|t^{-\alpha/N})$ in $\R^N$. This is a partial confirmation that the property of the infinite speed of propagation holds in higher dimensions for every solution of \eqref{model1} with $1<m<2$.

\noindent $\bullet$ \textbf{ Explicit solutions.} Y. Huang reports  \cite{Huang2013} the explicit expression of the Barenblatt solution for the special value of  $m$, $m_{ex} = (N+6s-2)/(N+2s)$. The profile is given by
\begin{equation*}
F_M(y) =\lambda\,(R^2 + |y|^2)^{-(N+2s)/2},
\end{equation*}
where the two constants $\lambda$ and $R$ are determined by the total mass $M$ of the solution  and the parameter $\beta$. Note that for $s=1/2$ we have $m_{ex}=1$, and the solution corresponds to the linear case, $u_t=(-\Delta)^{1/2} u$, $F_{1/2}(r)=C(a^2+r^2)^{-(N+1)/2}$.

\noindent$\bullet$ \textbf{ Different generalizations of model \eqref{ModelCaffVaz}} are worth studying:

\noindent(i) Changing-sign solutions for the problem $\displaystyle{\partial_t u= \nabla \cdot (|u| \nabla p), \quad p=(-\Delta)^{-s}u.}$

\noindent (ii) Starting from the Problem \eqref{ModelCaffVaz}, an alternative is to consider the problem
\begin{equation*}\label{modelKarch}
u_t=\nabla\cdot(|u|\nabla (-\Delta)^{-s}(|u|^{m-2}u)), \quad x \in \RN, \ t>0,
\end{equation*}
with $m>1$. This problem has been studied by Biler, Imbert and Karch in \cite{BilerImbertKarch}. They construct explicit compactly supported self-similar
solutions which generalize the Barenblatt profiles of the PME. In a later work by Imbert \cite{Imbert}, finite speed of propagation is proved for general solutions.

\noindent(iii) We should consider combining the above models into $ \displaystyle \partial_tu=\nabla(|u|^{m-1}\nabla p),$ $p=(-\Delta)^{-s}u. $ \\ When $s=0$ and $m=2$ we obtain the signed porous medium equation  $ \partial_tu=\Delta( |u|^{m-1}u)$.

\

\

\noindent {\bf \large Acknowledgments.}

\noindent  Authors partially supported by the Spanish project MTM2011-24696. The second author is also supported by a FPU grant from MECD, Spain.

\bibliographystyle{siam}

\end{document}